\documentclass[reqno]{amsart}
\usepackage{amsfonts, amssymb, amsmath, amsthm, diagrams, setspace}
\usepackage{graphics}
\usepackage{psfrag}
\newcommand{\R}{\mathbb{R}}
\newcommand{\C}{\mathbb{C}}
\newcommand{\h}{\mathbb{H}}

\newcommand{\cp}{\mathbb{CP}}
\newcommand{\Z}{\mathbb{Z}}

\newcommand{\s}{\mathcal{S}}
\newcommand{\m}{\mathcal{M}}
\newcommand{\mo}{\mathcal{M}^\circ}
\newcommand{\mC}{\mathcal{M}_{\mathbb{C}}}
\newcommand{\mCo}{\mathcal{M}_{\mathbb{C}}^\circ}
\newcommand{\n}{\mathcal{N}}

\newcommand{\mrs}{M_{\mc{R},\s}}
\newcommand{\ms}{M_\s}
\newcommand{\mf}[1]{\mathfrak{#1}}
\newcommand{\mc}[1]{\mathcal{#1}}

\newcommand{\tn}[1]{\textnormal{#1}}
\newcommand{\tb}[1]{\textbf{#1}}

\newtheorem{T}{Theorem}[section]
\newtheorem{TP}[T]{Proposition}
\newtheorem{TL}[T]{Lemma}
\newtheorem{TC}[T]{Corollary}
\theoremstyle{definition} \newtheorem{TE}[T]{Example}
\theoremstyle{definition}  \newtheorem{TR}[T]{Remark}
\theoremstyle{definition}  \newtheorem{TD}[T]{Definition}
\theoremstyle{definition}  \newtheorem*{Acknowledgements}{Acknowledgements}
\theoremstyle{plain} \newtheorem*{TheoremA}{Theorem A}
\theoremstyle{plain} \newtheorem*{TheoremB}{Theorem B}
\theoremstyle{plain} 

\newarrow {Dashto} {}{dash}{}{dash}>
\newarrow {Implies} ===={=>}

\begin{document}
\title{Compact anti-self-dual orbifolds with torus actions}
\author{Dominic Wright}
\email{domawright@gmail.com}
\maketitle
\begin{abstract}
We give a classification of toric anti-self-dual conformal structures on compact 4-orbifolds with positive Euler characteristic. Our proof is twistor theoretic: the interaction between the complex torus orbits in the twistor space and the twistor lines induces meromorphic data, which we use to recover the conformal structure. A compact anti-self-dual orbifold can also be constructed by adding a point at infinity to an asymptotically locally Euclidean (ALE) scalar-flat K\"ahler orbifold. We use this observation to classify ALE scalar-flat K\"ahler 4-orbifolds whose isometry group contain a 2-torus.
\end{abstract}

\section{Introduction}
A Riemannian curvature tensor on a 4-manifold can be decomposed into two components: the Ricci curvature and the \emph{Weyl curvature}. The latter of these is conformally invariant; therefore, a conformal structure has a well-defined Weyl curvature. The Weyl curvature can be decomposed further into self-dual and anti-self-dual components using the Hodge star operator, and a conformal structure is referred to as \emph{anti-self-dual} if the self-dual component vanishes.

For example, let $\big(P(x,y), Q(x,y)\big)$  be a pair of $\R^2$-valued functions satisfying both $P_1 Q_2 - P_2 Q_1 > 0$ and the linear differential equations
\begin{equation}\label{Equation: Joyce equation intro}P_x = Q_y,\quad P_y + Q_x = y^{-1}P\end{equation} over the half-plane $(y > 0)$, then
$$\frac{\tn{d}x^2 + \tn{d}y^2}{y^2} + \frac{(P_2 \tn{d}\theta_1 - P_1 \tn{d}\theta_2)^2+(Q_2 \tn{d} \theta_1 - Q_1 \tn{d} \theta_2)^2}{(P_1Q_2 - P_2Q_1)^2}$$
is an anti-self-dual metric on a 2-torus bundle over the half-plane. These \emph{torus symmetric} examples were discovered by Joyce in \cite{J95} and, by solving (\ref{Equation: Joyce equation intro}) with appropriate boundary conditions, he used them to construct anti-self-dual conformal structures on $\overline{\cp}^2\#\ldots\#\overline{\cp}^2$.

The classification of anti-self-dual conformal structures on compact 4-manifolds with isometry group containing a 2-torus was completed by Fujiki in \cite{F00}. He showed that, when such manifolds have positive Euler characteristic, they are diffeomorphic to $\overline{\cp}^2\#\ldots\#\overline{\cp}^2$ and their conformal structures coincide with those constructed by Joyce in \cite{J95}. Fujiki's proof exploits the \emph{twistor correspondence} between anti-self-dual manifolds and a class of complex 3-manifolds, known as \emph{twistor spaces}. This correspondence was developed by Atiyah, Hitchin and Singer in \cite{AHS78} following Penrose's original formulation for Minkowski space in \cite{P75}.

A compact 4-manifold with torus action is diffeomorphic to $\overline{\cp}^2\#\ldots\#\overline{\cp}^2$ precisely when it is \emph{simply-connected} and has \emph{negative-definite} intersection form. These properties are well-defined in the context of orbifolds; moreover, 4-orbifolds satisfying both these properties, and their quotients by finite subgroups of the torus action, admit anti-self-dual conformal structures arising from Joyce's construction. Thus, our first theorem provides a straightforward generalization of Theorem 1.1 of \cite{F00} from manifolds to orbifolds.

\begin{TheoremA}\label{Theorem: Theorem A}
Let $M$ be a compact 4-orbifold with positive orbifold Euler characteristic. If $M$ admits an anti-self-dual conformal structure with isometry group containing a 2-torus then:
\begin{enumerate}
\item{there is a torus equivariant diffeomorphism between $M$ and the quotient of a simply-connected 4-orbifold by a finite subgroup of the torus action;}
\item{$M$ has negative-definite intersection form;}
\item{and the conformal structure on $M$ is equivalent to one arising from Joyce's construction in \cite{J95}.}
\end{enumerate}
\end{TheoremA}

In \cite{J95} Joyce found \emph{K\"ahler representatives} of his conformal structures defined away from a torus fixed point. Moreover, these metrics are \emph{scalar-flat}, owing to the well-known result that a K\"ahler metric is anti-self-dual if and only if the scalar curvature vanishes (see for instance \cite{LS93}). The metrics constructed in this way are asymptotically locally Euclidean (ALE); thus, using Theorem A and a procedure suggested by Chen, LeBrun and Weber in \cite{CLW07}, we are able to prove the next theorem.

\begin{TheoremB}\label{Theorem: Theorem B}
Let $M$ be a compact anti-self-dual orbifold with positive orbifold Euler characteristic whose isometry group contains a 2-torus. Then there is an ALE scalar-flat K\"ahler representative of the conformal class away from a torus fixed point. Moreover, every scalar-flat K\"ahler 4-orbifold that is ALE to order $l > 3/2$ and has isometry group containing a torus arises in this way.
\end{TheoremB}

In the smooth case, surfaces satisfying Theorem B arise naturally when considering extremal K\"ahler metrics, such as in \cite{CLW07}, \cite{L91} and \cite{RS05}. Although implicit in the work of Joyce \cite{J95}, Calderbank and Singer \cite{CS04} constructed these scalar-flat K\"ahler metrics explicitly and showed that they were biholomorphic to toric resolutions of $\C^2/\Gamma$, for some cyclic $\Gamma \subset U(2)$. Thus, we have the following corollary.

\begin{TC}\label{Corollary: Smooth SFK}
Let $X$ be a smooth K\"ahler toric surface that is scalar-flat and ALE to order $l > 3/2$. Then, $X$ is biholomorphic to a toric resolution of $\mathbb{C}^2/\Gamma$ for some cyclic subgroup $\Gamma\subset U(2)$. Furthermore, $X$ is isometric to one of the scalar-flat K\"ahler metrics constructed by Calderbank and Singer in \cite{CS04}.
\end{TC}

We begin in Section \ref{Section: Compact Toric 4-Orbifolds} with a description of compact toric 4-orbifolds based on Orlik-Raymond \cite{OR70} and Haefliger-Salem \cite{HS91}. Then in Section \ref{Section: Definitions} we review quaternionic geometry and the quaternionic quotient, as well as twistor spaces and related definitions.


Our first step towards proving Theorem A is to establish that the closure of the complex torus orbits in the twistor space are complex orbifolds, this is done in sections \ref{Section: Local model} and \ref{Section: Analyticity of orbit closures}, and the method of proof is based on sections 4 and 5 of \cite{F00}. Next we focus our attention on twistor lines about which the torus action is free, which we will refer to as \emph{principal lines}. At this stage our approach diverges from Fujiki's and is motivated by the work of Donaldson and Fine in \cite{DF06}. They showed that a holomorphic involution and a $\C^2$-valued holomorphic function could be used to construct the twistor space of the germ of a toric anti-self-dual 4-manifold. In sections \ref{Section: The Involution} and \ref{Section: The meromorphic data} we construct a \emph{holomorphic involution} and a pair of \emph{meromorphic functions} on a principal line. The linearization of this data about a principal line can be thought of as a special case of Donaldson and Fine's holomorphic data. Our meromorphic data can be used to recover the conformal structure of a dense open subset in an anti-self-dual 4-orbifold. Then, in sections \ref{Section: The Riemann surface R} and \ref{Section: Uniqueness}, we extend this conformal structure using a correspondence between real torus invariant divisors in the twistor space and torus orbits in the 4-orbifold. At a local level, this correspondence is an example of a more general correspondence for anti-self-dual manifolds with torus symmetry, referred to as a \emph{microtwistor} correspondence \cite{C09}, in analogy with Hitchin's \emph{minitwistor} correspondence \cite{H82}. By this point we have established part (i) of Theorem A and, in Section \ref{Section: Theorems A}, we complete the proof of Theorem A and deduce Theorem B.

\begin{Acknowledgements}
This work forms part of the author's PhD thesis at Imperial College London, funded by the Engineering and Physical Sciences Research council.
I would like to thank my thesis advisor Simon Donaldson for useful discussions during the preparation of this article. Thanks also to David Calderbank, Joel Fine, Dmitri Panov, Stuart Hall and Richard Thomas for their helpful comments.
\end{Acknowledgements}

\section{Compact 4-orbifolds with torus actions}\label{Section: Compact Toric 4-Orbifolds}
\subsection{Orbifolds}\label{Subsection: Orbifolds}
Orbifolds are generalizations of manifolds, attributed to Satake \cite{S56} and Thurston \cite{T78}, with singularities arising locally from taking the
quotient by a discrete group. Let $M$ be a Hausdorff topological space, then a smooth n-dimensional orbifold structure can be defined by the following data:
\begin{enumerate}
\item{an open cover $\{U_i\}$ of $M$ that is closed under finite intersections;}
\item{an \emph{orbifold chart} associated to each $U_i$, which consists of an open subset $\tilde U_i \subset \R^n$ invariant under the effective action of a finite group of diffeomorphisms $\Gamma_i$, and a $\Gamma_i$-invariant continuous map $\varphi_i: \tilde U_i \rightarrow U_i$ that induces a homeomorphism from $\tilde U_i/\Gamma_i$ to $U_i$;}
\item{associated with each inclusion $U_i \subset U_j$, there is a smooth embedding $\psi_{ij} : \tilde U_i \rightarrow \tilde U_j$ and an injective homomorphism $f_{ij}: \Gamma_i \rightarrow \Gamma_j$ such that $\psi_{ij}$ is $f_{ij}$-equivariant and satisfies $\varphi_i = \varphi_j\circ\psi_{ij}$.}
\end{enumerate}

For a point $p \in U_i$ the subgroup of $\Gamma_i$ fixing $q \in \varphi^{-1}(p)\subset \tilde U_i$ is referred to as the \emph{orbifold structure group} of $p$ and the non-smooth or \emph{orbifold points} are those $p$ where the orbifold structure group is non-trivial.

We will be concerned with \emph{smooth} and \emph{effective} torus actions on $M$. We refer to an action of a torus $T^k$ on $M$ as effective if, for any $\theta \in T^k$, there exists $x \in M$ such that $\theta.x \neq x$. A smooth action of $T^k$ is defined to be a continuous action of $T^k$ on the underlying topological space $|M|$ satisfying the following two conditions \cite{HS91}:
\begin{enumerate}
\item{for each $\theta_0 \in T^k$, $x_0 \in M$ there are orbifolds charts $$\varphi: \tilde U \rightarrow U,\quad \varphi': \tilde U' \rightarrow U'$$ about $x_0$ and $\theta_0.x_0$ respectively, together with an open neighbourhood $A$ of $\theta_0$ in $T^k$, and a smooth map $$A\times \tilde U \rightarrow \tilde U': x \mapsto \theta.x$$ such that $\theta.\varphi(x) = \varphi'(\theta.x)$, }
\item{and, for each $\theta \in T^k$, the map $x \mapsto \theta.x$ is a diffeomorphism.}
\end{enumerate}

\subsection{Tensors on orbifolds}\label{Subsection: Tensors}
We now briefly outline how some standard structures on smooth manifolds, which we will make use of later, extend to the context of orbifolds. For a point $p$ in an orbifold $M$ with chart $\varphi:\tilde U_p \rightarrow U_p$ and structure group $\Gamma_p$, we define the orbifold tangent space at $p$ to be the quotient of $T\tilde U_p$ by the induced action of $\Gamma_p$. The tangent bundle of $M$ is then defined to be the union of the tangent spaces for all $p \in M$ equipped with transition functions that are induced from the diffeomorphisms $\psi_{ij}$. In analogy with manifolds the tensor bundles and subbundles can be defined in terms of the tangent bundle. The sections of these bundles that we are familiar with in the smooth case, such as Riemannian metrics and complex structures, are defined as sections over the orbifold covers $\tilde U_i$ that are invariant under the induced action of $\Gamma_i$ and agree on the overlaps. Moreover, these definitions allow us to generalize to the orbifold case geometric structures that are defined in the smooth case in terms of tensors, such as K\"ahler structures. The invariance under the action of the orbifold structure groups allows us to think of these structures as descending to $M$ in a well-defined manner.

Using these definitions we can define the Euler characteristic on an orbifold. Suppose that $X$ is a vector field on a compact orbifold $M$ that vanishes at isolated points. Let $p\in M$ be a zero of $X$ and let $\varphi: \tilde U \rightarrow U$ be a orbifold chart at $p$ with $\varphi(p) = 0$. We define the \emph{index} of $X$ at $p$ to be
$$\textnormal{ind}_p(X) := \frac{\textnormal{ind}_{0}(\varphi^* X)}{|\Gamma_p|}.$$ Then the index of $X$ on $M$ can be defined as the sum of the indices over all the zeros of $X$. Therefore, it is a rational number, which we will denote by $\textnormal{ind}_{orb}(X)$.

The \emph{orbifold Euler characteristic} was defined by Satake in \cite{S57} using a triangulation, in analogy with the definition of the topological Euler characteristic of a manifold. It follows that the orbifold Euler characteristic of $M$ is a rational number, which we will denote by $\chi_{orb}(M)$. Using this definition Satake extended the Poincar\'{e}-Hopf theorem to orbifolds:
$$\chi_{orb}(M) = \textnormal{ind}_{orb}(X).$$

\subsection{Orbit types}\label{Subsection: Compact 4-orbifolds} Let $M$ be a compact oriented 4-orbifold and let $F$ be a 2-torus acting smoothly and effectively on $M$. We will denote the Lie algebra of $F$ by $\mf{f}$ and the corresponding lattice by $\Lambda$, so that $F = \mf{f}/\Lambda$. We choose coordinates $(\theta_1, \theta_2)$ on $F$ that are linear on $\mf{f}$, with $0\leq \theta_1, \theta_2 < 1$. Thus, $\Lambda$ is generated by $(1,0)$ and $(0,1)$. We will use this notation throughout.

As shown in \cite{HS91}, the quotient of $M$ by the action of $F$ is a compact 2-orbifold with boundary $\partial N$ and interior $N^{\circ}$. The orbits over $\partial N$ are either isolated fixed points or have stabilizer subgroup $S^1$, while the orbits over $N^{\circ}$ are either free or isolated orbits with finite stabilizer. So there are four types of $F$-orbit:
\begin{description}
\item [\textbf{Principal orbits}] where the action of $F$ is free;
\item [\textbf{Exceptional orbits}] where the stabilizer is non-trivial and finite;
\item [\textbf{Boundary orbits}] where the stabilizer is isomorphic to $S^1$;
\item[\textbf{Fixed points}] where the stabilizer is $F$.
\end{description}
In this terminology, $N^\circ$ consists of principal orbits and isolated exceptional orbits, while $\partial N$, which has a finite number of components, consists of boundary orbits punctuated by isolated fixed points. Some examples of this decomposition are illustrated in Figure \ref{Figure: orbit spaces}.

\subsection{Boundary orbits and fixed points}\label{Subsection: Boundary orbits and fixed points}
We now describe the orbifold charts about a boundary orbit and then about a fixed point following \cite{HS91}. Let $x\in M$ be contained in a boundary orbit, let $$\varphi: \tilde U \rightarrow U \cong \tilde U/\Gamma$$ be an orbifold chart about $x$ and let $\tilde x = \varphi^{-1}(x)$. The compatibility of the orbifold chart with the action of $F$ implies that there is a smooth and effective 2-torus $\tilde F$ acting on $\tilde U$, whose generators span $\tilde \Lambda \subset \Lambda$, where $\Lambda/\tilde \Lambda = \Gamma$. There is an $S^1$-subgroup of $\tilde F$ stabilizing $\tilde x$, which corresponds $u \in \tilde \Lambda$. If we write $u = p.u'$, where $p$ is a positive integer and $u'$ is a primitive element of $\Lambda$, then $u'$ generates $\Gamma \cong \Z_p$.

This procedure associates some $u\in \Lambda$ to each boundary orbit. Since $u$ generates the same subgroup as $-u$, the correspondence between $x$ and $u$ is only well-defined up to sign. To avoid this ambiguity, we restrict these lattice points to a $180^\circ$ sector in $\mf{f}$ by assuming that either $u = (p,0)$ or $(0,1).u > 0$. Since $\Lambda$ is a discrete set, a continuity argument can be applied to conclude that there is a unique $u \in \Lambda$ associated to the connected component of the boundary orbits containing $x$.

For example suppose that $u = (p,0)$. We define $$\tilde U := \{(z_1,z_2)\in \C^2: |z_1|<\epsilon, |z_2 - 1| < \epsilon\},$$ and let $\tilde F$ act on $\tilde U$ as $$(\tilde \theta_1,\tilde \theta_2):(z_1,z_2) \rightarrow (e^{2\pi i\tilde \theta_1} .z_1, e^{2 \pi i\tilde \theta_2} .z_2).$$ An orbifold chart about $x$ can be written as $$\varphi:\tilde U \rightarrow U; (z_1, z_2) \mapsto (z_1^p, z_2),$$ where $\tilde x$ is identified with $(0,1) \in \tilde U$.

We now give a local model about $F$-fixed points that is determined by these lattice points. As explained in Subsection \ref{Subsection: Compact 4-orbifolds}, $\partial N$ consists of boundary orbits and isolated fixed points. Thus, a neighbourhood of a fixed point $x\in M$ contains boundary orbits associated with the linearly independent lattice points  $u$ and $u'$. We will now denote the lattice with generators $u,u'$ by $\tilde \Lambda$ and the corresponding torus by $\tilde F$. Then we define $\Gamma = \Lambda/\tilde \Lambda$ and so, $F = \tilde F/\Gamma$.

Let $$\tilde U := \{ (z_1,z_2) \in \C^2: r_1 <\epsilon, r_2 < \epsilon \}$$
and let $\tilde F$ act on $\tilde U$ by
$$(\tilde\theta_1,\tilde\theta_2):(z_1,z_2) \rightarrow (e^{2 \pi i\tilde\theta_1} .z_1, e^{2 \pi i\tilde\theta_2} .z_2).$$
The action of $\Gamma \subset \tilde F$ on $\tilde U$ induces a quotient map $$\varphi: \tilde U \rightarrow U = \tilde U/\Gamma,$$ and the action of $\tilde F$ on $\tilde U$ descends to an action of $F$ on $U$. If we identify $x$ with $0 \in U$, then $\varphi: \tilde U \rightarrow U$ is an orbifold chart about $x$. The orbifold structure group of $x$ is $\Gamma$; thus, $x$ is smooth if and only if $u$ and $u'$ generate $\Lambda$.

\subsection{Simply connected orbifolds}\label{Subsection: M simply
connected}
It follows from the work of Haefliger and Salem in \cite{HS91} that $M$ is simply connected ($\pi_1^{orb}(M) = 0$) when:
\begin{enumerate}
\item{the orbit space $N$ is topologically a closed disc;}
\item{there are no exceptional orbits;}
\item{and the elements of $\Lambda$ associated to the components of the boundary orbits generate $\Lambda$.}
\end{enumerate}

An orientation can be given to $N^{\circ}$, which induces an orientation on $\partial N$; accordingly, we can label the fixed points in $M$ as $x_1, \ldots, x_k$. The element in $\Lambda$ associated with the connected component of the boundary orbits between $x_i$ and $x_{i+1}$ will be labeled $u_i$. Thus, associated with $M$ and the singled out fixed point $x_1$ there is an ordered set $$\s := \{u_1,\ldots,u_n\} \subset \Lambda.$$

Similarly to the simply connected case, a compact 4-orbifold with smooth and effective torus action, which satisfies conditions (i) and (ii), determines an ordered set $\s\subset \Lambda$. If $\s$ generates the sublattice $\Lambda_\s \subset \Lambda$, then $M$ has a universal cover with deck transformation group $\Gamma_\s = \Lambda/\Lambda_\s$.

In \cite{HS91} it was shown that $M$ is determined, up to an $F$-equivariant diffeomorphism, by $\s$. Thus, a diffeomorphism between two compact toric 4-orbifolds of this form is equivalent to a \emph{weighted diffeomorphism} of the orbit spaces, by which we mean a diffeomorphism that respects the orbit types. 

The compact orbifold with an action of $F$ that determines $\s$, in the way described above, will be denoted by $M_\s$. Note that there is a torus equivariant diffeomorphism between $M_\s$ and $M_{\s'}$ if some cyclic permutation of $\s'$ can be transformed to $\s$ under the action of $SL(2,\Z)$.

\begin{figure}[htbp]
  \psfrag{x}{$\zeta_1$}
  \psfrag{y}{$\zeta_3$}
  \psfrag{z}{$\zeta_2$}
  \psfrag{p}{$(1,0)$}
  \psfrag{q}{$(p,1)$}
  \psfrag{r}{$(0,1)$}
  \psfrag{s}{$(1,0)$}
  \psfrag{t}{$(0,1)$}
  \psfrag{a}{$\zeta_1$}
  \psfrag{b}{$\zeta_4$}
  \psfrag{c}{$\zeta_3$}
  \psfrag{d}{$\zeta_2$}
  \psfrag{A}{$(2,0)$}
  \psfrag{B}{$(0,1)$}
  \psfrag{C}{$(2,0)$}
  \psfrag{D}{$(0,1)$}
  \psfrag{X}{$\cp^2_{1,1,p}$}
  \psfrag{Y}{Hopf surface}
  \psfrag{Z}{$(S^2\times S^2)/\Z_2$}
 \centering \scalebox{0.8}{\includegraphics{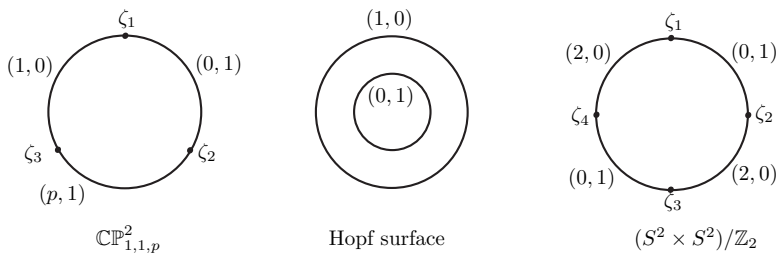}}
  \caption{Orbit spaces of compact toric 4-orbifolds} \label{Figure: orbit spaces}
\end{figure}

\begin{TE}\label{Example: Toric orbifolds}
Consider $\R^4\cong \C^2$ with the action of $F$ given by $$(\theta_1, \theta_2): (z_1, z_2)\mapsto (e^{2 \pi i\theta_1}z_1, e^{2 \pi i\theta_2}z_2).$$ The quotient can be identified with the quarter plane, and the two edges have stabilizer subgroups $(1,0)$ and $(0,1)$ in $\Lambda$. By adding a fixed point at infinity, we obtain the compact manifold $S^4$ with an action of $F$ that can be defined by $\s = \{(1,0), (0,1)\}$. In Figure \ref{Figure: orbit spaces} the only simply connected example is $\cp^{2}_{1,1,p}$. The Hopf surface does not fit into the framework described in this subsection, as it has fundamental group isomorphic to $\Z$ and so, the universal cover is not compact. The example $(S^2\times S^2)/\Z_2$ can be constructed as the quotient of the simply connected 4-manifold $S^2\times S^2$ with respect to the action of $\Z_2$ generated by $(\frac{1}{2},0)\in F$.
\end{TE}

\section{Twistor spaces}\label{Section: Definitions}
\subsection{Twistor spaces}\label{Subsection: Twistor spaces}
We define a \emph{twistor line} in a complex 3-manifold to be a projective line $L$ with normal bundle $$N_L \cong \mathcal{O}(1)\oplus\mathcal{O}(1),$$ and we define a \emph{twistor space} to be a complex 3-manifold containing a twistor line. Sections of $N_L$ correspond to infinitesimal deformations of $L$. Since $H^1(L,N_L) = 0$, Kodaira's deformation theorem \cite{KS58} tells us that genuine deformations of $L$ are parameterized by a complex 4-manifold $M_\C$ with $T_L M_\C \cong H^0(L,N_L)$. Moreover these deformations have isomorphic normal bundles, since $H^1(L,\tn{End}(N_L)) = 0$.

These twistor lines can be used to give $M_\C$ a natural conformal geometry. We refer to a section of $N_{L}$ with a zero as a null vector; thus, we are characterizing null vectors as infinitesimal deformations intersecting $L$. Since a section of $N_{L}$ can be written as $(az_0 +bz_1, cz_0+dz_1)$, the null vectors are determined by the equation $ad-bc = 0$. This equation defines a \emph{null cone}, which is enough to specify a conformal structure on $T_{L}M_\C$ and hence, a conformal structure on $TM_\C$.

An anti-holomorphic involution $\gamma$ on a twistor space is referred to as a \emph{real structure}, if $\gamma$ has no fixed points, and the $\gamma$-fixed twistor lines are referred to as \emph{real twistor lines}. If a twistor space is equipped with a real structure, then the set of real twistor lines is a real 4-manifold with a conformal structure induced from $TM_\C$.

Let $Z$ be a twistor space and let $M$ be the conformal 4-manifold constructed in the manner described. We define the \emph{twistor fibration} to be the projection
$$\pi: Z \rightarrow M$$ mapping real twistor lines to points in $M$. The cornerstone of the twistor correspondence for Riemannian metrics is the equivalence between integrability of the complex structure on $Z$ and anti-self-duality of the conformal structure on $M$, \cite{AHS78}.

We now outline how $Z$ can be constructed from $M$. Let $M$ be an oriented conformal 4-manifold and let $J_x$ be an almost complex structure on the tangent space at $x\in M$. We can fix a choice of metric in the conformal structure on $T_xM$ and, with respect to this, we can choose the unit vectors $v \in T_x M$ and $$w \in \langle \{ v, J_x v \}\rangle^{\perp}.$$ We say that $J_x$ is compatible with the metric and the orientation, if $$\{v, J_x v, w, J_x w\}$$ forms an oriented orthonormal basis of $T_x M$. On $T_x M$ we can choose three almost complex structures $I,J,K$ compatible with the metric and the orientation such that $$\{v, I v, J v, K v\}$$ is an orthonormal basis. These almost complex structures can be used to construct a $\cp^1$-bundle $Z$ over $M$: the \emph{bundle of almost complex structures} on $M$ compatible with the metric and orientation. There is a natural almost complex structure on $Z$ and, if the conformal structure on $M$ is anti-self-dual, then $Z$ is a complex manifold. Moreover, the fibers of $Z$ are twistor lines and $Z$ is a twistor space.

\subsection{The twistor lift}\label{Subsection: Twistor lifts}
Let $M$ be an oriented conformal 4-manifold and let $$\phi: N \rightarrow M$$ be an immersion of an oriented surface. The conformal structure on $M$ induces a conformal structure on $N$. Since $N$ is an oriented surface, this conformal structure induces two complex structure $\pm J$ on $N$. For $n\in N$, let $J_n$ be the compatible complex structure on $T_{\phi(n)}M$ that restricts to $J$ on $T_{\phi(n)} N$. We define a \emph{twistor lift} \cite{ES85} of $N$ to $Z$ as
$$\tilde\phi_+: N \rightarrow Z;\; n \mapsto J_n.$$
This map can be broken down into two components: a map sending $n$ to the 2-plane $T_{\phi(n)}N \subset T_{\phi(n)} M$, followed by a map sending $T_{\phi(n)}N$ to $J_n$. The significance of this decomposition is that the first component is the \emph{Gauss map} for a surface immersed in a conformal 4-manifold. Thus, $\phi_+$ can be thought of as the projection of the Gauss map to the twistor space. Similarly, we can define the twistor lift $\phi_{-}$ with respect to $-J$ on $N$, and we will denote the image of $\phi_{\pm}$ by $N_{\pm}$. It follows from these definitions that $N_{+}$ and $N_{-}$ are interchanged by the real structure $\gamma$.

\begin{figure}[htbp]
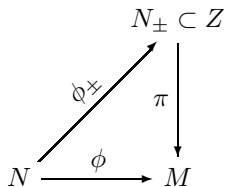

\begin{diagram}
         &                       &     N_{\pm} \subset Z    \\
         &  \ruTo^{\phi_{\pm}}  &       \dTo^{\pi}        \\
    N    &    \rTo^{\phi}        &          M                \\
\end{diagram}
  \caption{The twistor lift}
\end{figure}

If $\phi$ is an immersion of an unoriented surface, then the twistor lift can be defined locally in the manner described above, since there is no global choice of sign for a complex structure on $N$. Extending this lift produces an unbranched double cover of $N$ in $Z$, on which the real structure restricts to an anti-holomorphic involution.

The twistor lift is holomorphic if its image is a complex curve in $Z$. In this case, the restriction of $\pi$ to this curve is a conformal map onto $\phi(N)$. Conversely, a complex curve in $Z$, which intersects with twistor lines transversally, is the holomorphic twistor lift of an immersed surface in $M$.

\subsection{Orbifold twistor spaces}\label{Subsection: Orbifold twistor
space}
The definitions in the preceding subsections have been made in terms of manifolds, rather than orbifolds, with the intention of simplifying the exposition. However, these definitions are formulated in terms of tensors and tensor bundles and so, they can be extend as per Subsection \ref{Subsection: Tensors}. In particular, a twistor space can be described in terms of tensors: as the bundle of almost complex structures that are compatible with both the orientation and the conformal structure. Next, we will present the definition of the twistor space of an anti-self-dual orbifold in more detail.

Let $M$ be an anti-self-dual orbifold with charts $\tilde{U_i}\rightarrow U_i := \tilde{U_i}/\Gamma_i$. Since the manifolds $\tilde{U_i}$ are equipped with $\Gamma_i$-invariant anti-self-dual metrics, their twistor spaces, which we will denote by $\tilde{V_i}$, have $\Gamma_i$-invariant complex structures. The twistor space of $U_i$ is defined as the complex orbifold $V_i := \tilde{V_i}/\Gamma_i$, where $\Gamma_i$ acts by biholomorphisms on $\tilde{V_i}$ that are induced from the isometric action of $\Gamma_i$ on $\tilde{U_i}$. The $\Gamma_i$-invariant diffeomorphisms between the $\tilde{U_i}$, with which we ``glue" the $U_i$, then induce $\Gamma_i$-invariant biholomorphisms between the $\tilde{V_i}$. Thus, we have an atlas of charts defining a complex orbifold. This orbifold is defined to be the twistor space of $M$, and we will denote it by $Z$. Note that an orbifold point in $Z$ can only arise on a twistor line over an orbifold point in $M$, since a point $p$ with isotropy group $\Gamma_p$ corresponds to a twistor line $L_p \cong \cp^1/\Gamma_p$.

\section{An example}
\subsection{Gravitational Instantons}
In Section \ref{Subsection: Theorem B} we explain how to compactify asymptotically locally Euclidean (ALE) scalar-flat K\"ahler 4-orbifolds by adding one point at infinity, thus obtaining anti-self-dual 4-orbifolds. In this example we apply this procedure to ALE hyperK\"ahler 4-manifolds, although we defer the details of the compactification until Subsection \ref{Subsection: Theorem B}. Note that this procedure is applicable because hyperK\"ahler manifolds coincide with Ricci-flat K\"ahler manifolds in four dimensions and so, they are a special case of scalar-flat K\"ahler manifolds.

In \cite{GH78} Gibbons and Hawking constructed a family of ALE hyperK\"ahler metrics on the minimal resolution of the orbifold $\C^2/\Gamma$, where $\Gamma \subset SU(2)$ is cyclic of order $k \geq 1$. Each metric in this family has $S^1$-symmetry and there is a sub-family, with rotational symmetry about the exceptional set, that has 2-torus symmetry. When examples in this sub-family are compactified, the underlying 4-orbifold with torus action can be described in the terminology of Subsection \ref{Subsection: M simply connected} by the ordered set of combinatorial data $$\s : = \{(m_0,n_0), \ldots,(m_k,n_k)\}\subset \Lambda \cong \Z^2,$$ where $(m_i,n_i) = (1,k-i)$. We will denote this compact 4-orbifold with torus action by $M_\s$. (This observation can be verified using the toric description of the resolution, presented for instance in \cite{CS04}.)

We will denote the twistor space of one of these compact anti-self-dual 4-orbifolds by $Z$. In \cite{H79} Hitchin constructed $Z$ explicitly as a complex 3-orbifold bimeromorphic to a singular hypersurface $\widetilde{Z}$ in the total space of \begin{equation}\label{Equation: Hitchin Bundle}\mathcal{O}(k)\!\oplus\!\mathcal{O}(k)\!\oplus\!\mathcal{O}(2) \rightarrow \mathbb{CP}^1.\end{equation} The defining equation of $\widetilde{Z}$ can be written \begin{equation}\label{Hitchin Equation}xy = \prod_{i=1}^k(z - p_i),\end{equation} where $x\in\mathcal{O}(k)$, $y\in\mathcal{O}(k)$, $z\in\mathcal{O}(2)$ and $p_i$ are holomorphic sections of $\mc{O}(2)$. In terms of the affine coordinate $u$ on $\cp^1$, the twistor space of a toric example can be defined using $$p_i(u) = 2b_iu$$ for $i=1,\ldots, k$, where $b_i$ are real and ascending. In particular, $\{b_1,\ldots, b_k\}$ parameterizes the family of anti-self-dual metrics on $M_\s$.

Away from its singular set $\widetilde{Z}$ is biholomorphic to $Z$ and so, we will identify the two non-singular sets for the purposes of this example. As a consequence, the twistor lines are identified with sections of the bundle (\ref{Equation: Hitchin Bundle}) away from the singular set. So suppose we want to choose a twistor line corresponding to the holomorphic section $z(u) = u^2 - 1$ of $\mc{O}(2)$. Then we can factorize equation (\ref{Equation: Hitchin Bundle}) as $$xy = \prod_{i=1}^k (u - \beta_i^+)(u - \beta_i^-),$$ where $\beta_i^\pm = b_i \pm (b_i^2 +1)^{1/2}.$ Therefore, there is a twistor line $L$ given by \begin{equation}\label{Equation: Hitchin Section}(u, x(u), y(u), z(u)) = (u, \prod_{i=1}^k (u - \beta_i^+), \prod_{i=1}^k (u - \beta_i^-), u^2 - 1).\end{equation} In the next subsection we find a pair meromorphic functions defined on $L$.

\subsection{A meromorphic function}
The holomorphic action of $\C^*\times\C^*$ on the total space of (\ref{Equation: Hitchin Bundle}) given by $$(s,t):(u,x,y,z) \mapsto (su,tx,t^{-1}y,sz)$$ preserves $\widetilde{Z}$. This $\C^*\times\C^*$ action $\widetilde{Z}$ is induced from the 2-torus action on $M_\s$. If we choose a point $z \in L$ we can ask: does there exist non-trivial $(s,t) \in \C^*\times\C^*$ such that $(s,t).z \in L$?

A positive answer to this question involves finding $(s,t)$ such that $$(su,tx(u),t^{-1}y(u),sz(u)) = (su,x(su),y(su),z(su)) \in L,$$ or equivalently solving $$su^2 - s = sz(u) = z(su) = s^2u^2 - 1$$ and then setting $$t = x(su)/x(u).$$ Away from a finite number of points there is a single non-trivial solution $$(s,t) =  (-u^{-2}, x(-u^{-1})/x(u))$$ and this defines a pair of meromorphic function $s(u)$ and $t(u)$ associated with $L$. If we change affine coordinate on $L$ using $u = i(1-z)/(1+z)$, then we can use the formula for $x(u)$ given in equation (\ref{Equation: Hitchin Section}) to show that this pair can be written as
$$(s(z), t(z)) = \Big(\Big(\frac{z-1}{z+1}\Big)^2,\Big(\frac{z-1}{z+1}\Big)^k\prod_{i=0}^k\Big(\frac{z + z_i}{z - z_i}\Big)\Big),$$ where $z_i = (\beta_i^+ + i)/(\beta_i^+ - i)$. By the definition of the $\beta_i^+$ it follows that $|z_i| = 1$ for each $i$ and $0 < \tn{arg}(z_1) < \cdots < \tn{arg}(z_k) < \pi.$

\subsection{Encoding the topological and conformal structure}
We can rewrite $(s(z), t(z))$ as $$\Big(\prod_{i=0}^k\Big(\frac{z-z_i}{z + z_i}\Big)^{a_i},\prod_{i=0}^k\Big(\frac{z-z_i}{z + z_i}\Big)^{b_i}\Big),$$ where $z_0 = 1$, $(a_0, b_0) = (m_0,n_0) + (m_k,n_k)$ and $(a_i,b_i) = (m_i,n_i) - (m_{i-1}, n_{i-1})$ for $i = 1,\ldots, k$. Then, if we repeat this calculation using another twistor line about which the action of $\C^*\times\C^*$ is free, we obtain the same pair of functions, after an appropriate change in affine coordinate. Thus, there is a correspondence between $Z$ and the meromorphic functions $(s(z),t(z))$. This leads to two important observations: one is that the $b_i$, which parameterize the family of conformal structures on $M_\s$, are determined by the position of these poles and zeros; the other is that there is a one-to-one correspondence between the order of the poles and zeros of the combinatorial data $\s$. Thus, $(s(z),t(z))$ encodes both the topology of $M_\s$ and the conformal structure on $M_\s$.

So it is natural to ask: do all compact anti-self-dual orbifolds with torus actions correspond to pairs of meromorphic functions in this way? Then, a positive answer to this would prompt the second question: is this correspondence one-to-one or equivalently, can we recover the anti-self-dual orbifolds from these functions? In what follows we provide positive answers to both these questions under the (necessary) assumption that our orbifold has positive orbifold Euler characteristic: in Sections \ref{Section: Local model} and \ref{Section: Analyticity of orbit closures} we establish that the $\C^*\times\C^*$-orbits define complex orbifolds in the twistor space and this enables us to proceed with the calculation of $(s(z),t(z))$ in Sections \ref{Section: The Involution} and \ref{Section: The meromorphic data}; then in Sections \ref{Section: The conformal structure} to \ref{Section: Uniqueness} we use these functions to recover the conformal structure.

\section{A local model}\label{Section: Local model}
\subsection{A set $B \subset M$ containing a fixed point}\label{subsection: constructing B}
We will assume that $M$ is a compact 4-orbifold with $\chi_{orb}(M)>0$ equipped with an anti-self-dual conformal structure that is invariant under the action of a two torus. As in Section \ref{Section: Compact Toric 4-Orbifolds}, we denote this torus by $F$, its Lie algebra by $\mf{f}$ and its lattice of generators by $\Lambda \subset \mf{f}$.

Let $X$ be a vector field generating a generic $S^1$-subgroup of $F$. The points where $X$ vanishes are precisely the same as the fixed points of $F$ and so, any zeros of $X$ are isolated. Since $\chi_{orb}(M)>0$, the Poincar\'{e}-Hopf theorem for orbifolds (stated in Subsection \ref{Subsection: Tensors}) implies that $X$ must have zeros on $M$. Thus, the action of $F$ has fixed points.

In the terminology of Subsection \ref{Subsection: Compact 4-orbifolds}, the boundary of $N$ consists of $F$-orbits in $N$ where the action of $F$ has infinite stabilizer. In particular, it comprises boundary orbits and fixed orbits. Let $\partial N_0 \subset \partial N$ be a connected component containing a fixed orbit $\zeta_1$. The orientation on $N^{\circ}$ induces an orientation on $\partial N_0$; accordingly, we can label the remaining fixed orbits on $\partial N_0$ as $\zeta_{2}, \ldots, \zeta_k$. We will denote the preimage of $\partial N_0$ in $M$ by $B$ and the preimage of $\zeta_i$ by $x_i$, for $i = 1, \ldots, k$. The preimage of the closed segment of $\partial N_0$ between $\zeta_i$ and $\zeta_{i+1}$ is a compact 2-orbifold \cite{HS91}, which we will denote by $B_i$.

So, we have a cycle of surfaces $$B = \bigcup_{i=1}^n B_i,$$ which intersect transversely at $$B_{i-1}\cap B_{i} = \{x_{i}\}.$$ As explained in Subsection \ref{Subsection: Boundary orbits and fixed points}, each $B_i$ corresponds to a unique $u_i\in \Lambda$, where consecutive $u_i$ are linearly independent. Thus, we have an ordered set of combinatorial data $\s = \{u_1, \ldots, u_k\}$ associated to $B$. We will choose the coordinates for $F$ so that $u_1 = (p,0)$, and from now on we will use this fixed choice of coordinates for $F$. We will denote the orbifold structure group of non-fixed points on $B_i$ by $\Gamma^i$ and the orbifold structure group of $x_i$ by $\Gamma_i$, for $i = 1,\ldots, k$.

\subsection{A set $\Sigma \subset Z$ containing a fixed point}\label{Subsection: Describing a singular set}
Let $\pi:Z\rightarrow M$ be the twistor fibration with real structure $\gamma$. The action of $F$ on $M$ induces a natural action of $F$ on $Z$. We will denote the complexification of $F$ by  $G\cong\C^*\times\mathbb{C^*}$. Since $Z$ is compact, there is a biholomorphic action of $G$ on $Z$. Also note that $\pi$ and $\gamma$ commute with the action of $F$; moreover, \begin{equation}\label{Equation: gamma anticommutes}\gamma(g.z) = \bar{g}^{-1}.\gamma(z)\end{equation} for $g\in G$ and $z \in Z$.

We will denote the subset of $\pi^{-1}(B)$ where the action of $F$ is not free by $\Sigma$ and the twistor line corresponding to the fixed point $x_i$ by $L_i$. This line is preserved by the action of $F$, since $x_i$ is fixed by the action of $F$. Therefore, $L_i$ is contained in $\Sigma$ (as $F$ cannot act freely on $L_i$) and we will write $$\mc{L} := \cup_{i=1}^k L_i.$$

In what follows we will see that $\Sigma$ decomposes as $C\cup \mc{L},$ where $C$ denotes the twistor lift of $B$.
In subsections 4.5 and 4.6 of \cite{F00}, Fujiki described $C$ in some detail in the smooth case. In the orbifold case the picture is much the same, as we can pass to orbifold covers. We will briefly reinterpret this description in terms of the twistor lift and generalize it to orbifolds.

We will denote the set of non-fixed points in $B_i$ by $B_i'$ and the complexification of the stabilizer subgroup of points in $B_i$ by $G^i$. If $L_x$ denotes the twistor line of $x \in B_i'$, then the action of $G^i$ on $Z$ preserves $L_x$. Recall from Subsection \ref{Subsection: Twistor lifts} that there are two twistor lifts to $Z$ of an oriented surface in $M$ and these are interchanged by $\gamma$. So, the twistor lift maps $B_i$ at $x$ to a pair of complex structures, which we will denote by $x^\pm \in L_x$.

For a smooth point $x$, let $F_x \subset F$ denote the stabilizer of $x$. Then the induced action of $F_x$ on $T_xM$ acts trivially on $T_x B_i$. Since $T_x B_i$ is an eigenspace of the $F_x$-action, the 2-dimensional subspace orthogonal to $T_x B_i$ with respect to the conformal structure is also preserved by the action of $F_x$. A complex structure is fixed by $F_x$ if it preserves this splitting and so, the only complex structures preserved by $F_x$ are $x^\pm$. Furthermore, every other complex structure at $x$ has non-trivial stabilizer in $F$, as noted by Fujiki in Lemma 4.1 of \cite{F00}. If $x$ is an orbifold point, then the orbifold chart is given by taking the quotient of the smooth case by $\Gamma^i$. It follows that $x^\pm$ are orbifold points with structure group $\Gamma^i$, and the other points on $L_x$ are smooth. The next lemma summarizes these observations.

\begin{TL}\label{Lemma: Constructing C}
Let $x\in B_i'$ and let $L_x$ be the corresponding twistor line. There are two points on $L_x$ that are fixed by the action of $G^i$ and at all other points on $L_x$ the action of $F$ is free. These $G^i$-fixed points are interchanged by $\gamma$ and have orbifold structure group $\Gamma^i$.
\end{TL}

If we denote the union of the twistor lifts of $B_1, \ldots, B_k$ by $C$, then it is a direct consequence of lemma \ref{Lemma: Constructing C} that $$\Sigma = C\cup \mc{L}.$$  The next lemma establishes some details about $C$.

\begin{TL}\label{Lemma: Defining C}
The twistor lift of $B_i$ consists of a pair of curves, which we will denote by $C_i^\pm$. These curves satisfy $C_i^\pm \cap C_{i+1}^\pm = \{z_{i+1}^\pm\}$ for $i = 1, \ldots, k-1$, where $z_i^\pm \in L_i$ have stabilizer subgroup $G$ and orbifold structure group $\Gamma_i$. Furthermore, if we denote the non-fixed points on $C_i^\pm$ by ${C_i^\pm}'$, then points in ${C_i^\pm}'$ have stabilizer subgroup $G^i$ and orbifold structure group $\Gamma^i$.
\end{TL}
\begin{proof}
We will denote the two twistor lifts of $B_1$ at $x_1$ and $x_2$ by $z_1^\pm$ and $z_2^\pm$, respectively. The subset of points in $\pi^{-1}(B_i)$ stabilized by $G_i$ is a $\gamma$-invariant complex subvariety. Therefore, $C_1^\pm$ are complex curves in $Z$ interchanged by $\gamma$. 

We can decompose $F$ as the product of the stabilizers at $B_1$ and $B_2$. The isotropic representation of $F$ with respect to this splitting is equivalent to the decomposition of $T_{x_2}M$ into the tangent spaces to $B_1$ and $B_2$. These tangent spaces are eigenspaces of the $F$-action and so, they are orthogonal with respect to the orbifold conformal structure. Thus, the twistor lifts of $B_2$ at $x_2$ are the same as those for $B_1$ at $x_2$. Therefore, we can denote the twistor lifts of $B_2$ by $C_2^\pm$, where $$C_1^\pm \cap C_2^\pm = \{z_2^\pm\}.$$ Since $z_2^+$ and $z_2^-$ are fixed by the action of $G^1$ and $G^2$, they are fixed by $G$. Moreover, $z_2^\pm$ has the same orbifold structure group as $x_2$. We can continue inductively to define the curves $C_i^\pm$ for $i = 1, \ldots, k$, where $$C_i^\pm \cap C_{i+1}^\pm = \{z_{i+1}^\pm\},$$ for $i = 1, \ldots, k-1$.

Finally, we establish that the stabilizer subgroup of non-fixed points in $C_i^\pm$ is $G^i$. The action of $F$ preserves $B_i$ and so, the action of $G$ preserves $C_i^\pm$. Let $G'$ be the stabilizer of non-fixed points in $C_i^\pm$, so $G^i \subset G'$ and dim$G' = 1$. Then, $$G/G' \cong {C_i^\pm}' \cong \C^*.$$ Since $g \notin F$, $g \in G'\backslash G^i$ generates an infinite cyclic subgroup in $G/G_i$. Therefore, $G/G'$ is isomorphic to a complex torus. This gives a contradiction; thus, $G'  = G^i$.
\end{proof}

\begin{figure}[htbp]
\centering \scalebox{1.0}{\includegraphics{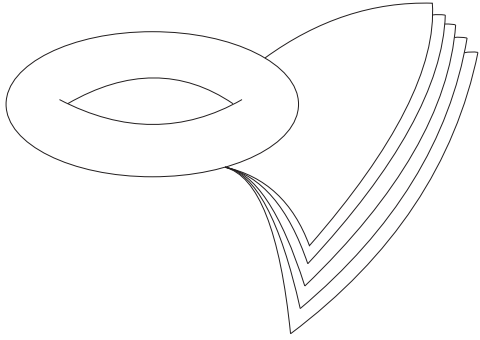}}
  \caption{What we want to avoid: a $G$-orbit intersecting $C$ in more than one leaf}\label{Figure: Torus Foliation}
\end{figure}

\begin{TR}
Our goal in this section and the next is to establish that the closure of the $G$-orbits that intersect $C$ are complex 2-orbifolds embedded in $Z$. In order to do this we first (in the next two subsections) establish a local model for a neighbourhood of points in $C$, based on the model for Euclidean twistor space. Then we show that the action of $G$ induces a foliation of these neighbourhoods in Subsection \ref{Subsection: Local analyticity}. Finally in Subsection \ref{Subsection: analyticity of orbit closures} we establish that each $G$-orbit intersects these neighbourhoods in a single leaf of the foliation.

What we intend to rule out is the possibility that a single $G$-orbit intersects one of these neighbourhoods in more than one leaf. An example in which $G$-orbits have this undesirable property is the Hopf-surface, the orbit space of which was given in Figure \ref{Figure: orbit spaces}. In this case $B$ must be a 2-torus and then $C$ must be a disjoint pair of complex curves in $Z$. As a result we cannot follow the proof of Lemma \ref{Lemma: Defining C} to obtain a contradiction and so, the stabilizer of $C$ may be larger than the complexification of the stabilizer of $B$. If we choose a leaf in the local foliation about a point in one of the tori in $C$, then we can follow a path around this torus, within the $G$-orbit containing that leaf, to arrive at a different leaf, see Figure \ref{Figure: Torus Foliation}. In analogy with an irrational circle action on a 2-torus, a single $G$-orbit contains infinitely many leaves in the Hopf-surface example.
\end{TR}

So, we can write the twistor lift of $B$ as $$C = \bigcup_{i=1}^{k} C_i^+ \cup \bigcup_{i=1}^{k} C_i^-.$$ Note that we have not yet determined whether $C_k^- \cap C_{1}^+ = \{z_{1}^+\}$ or $C_k^+ \cap C_{1}^+ = \{z_{1}^+\}.$ The former possibility implies $C$ is connected, while the latter implies $C$ consists of two connected components that are interchanged by $\gamma$. In Lemma \ref{Lemma: C connected}, we prove that $C$ is connected, as depicted in Figure \ref{Figure: C projects to B}. In this figure the M\"obius band represents $\pi^{-1}(B)$ and so, twistor lines are represented by lines crossing this band. The boundary of the band along with the lines $L_i$ represent $\Sigma$.

\begin{figure}[htbp]
\psfrag{a1}{$z_1^+$}\psfrag{a2}{$z_2^+$} \psfrag{a3}{$z_3^+$} \psfrag{a4}{$z_1^-$} \psfrag{a5}{$z_2^-$}\psfrag{a6}{$z_3^-$}\psfrag{c1}{$x^-$}\psfrag{c2}{$x^+$}\psfrag{A1}{$C_1^+$}\psfrag{A2}{$C_2^+$}\psfrag{A3}{$C_3^+$}\psfrag{A4}{$C_1^-$}\psfrag{A5}{$C_2^-$}\psfrag{A6}{$C_3^-$}\psfrag{L1}{$L_1$}\psfrag{L2}{$L_2$}\psfrag{L3}{$L_3$}\psfrag{L4}{$L_x$}\psfrag{b1}{$x_1$}\psfrag{b2}{$x_2$}\psfrag{b3}{$x_3$}\psfrag{d1}{$x$}\psfrag{B1}{$B_3$}\psfrag{B2}{$B_2$}\psfrag{B3}{$B_1$}\psfrag{p}{$\pi$}
  \centering \scalebox{0.8}{\includegraphics{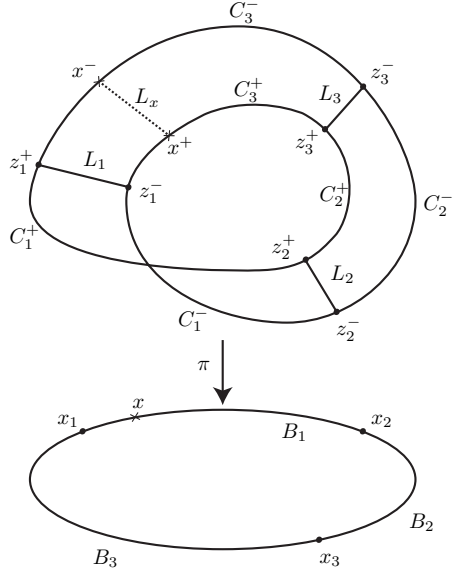}}
  \caption{The projection from $\Sigma \subset Z$ to $B \subset M$}\label{Figure: C projects to B}
\end{figure}

\subsection{The local model about fixed points in $\Sigma$}\label{Subsection: Local Model}
The next two subsections provide local models for the action of $G$ on neighbourhoods of points in $\Sigma$. Although the construction of these local models is rather involved, it is worth keeping in mind that, in principle, we are proving that $z$ has an orbifold chart that can be modeled on the following example.

\begin{TE}\label{Example: Twistor space}
Recall from Subsection \ref{Subsection: Twistor spaces}, that $\mc{O}(1)\oplus\mc{O}(1)$ is the twistor space of $\R^4$. A point can be added to $\R^4$ at infinity to obtain $S^4$. In the twistor space this corresponds to adding a twistor line to produce $\cp^3$. If we identify $S^4$ with $\mathbb{HP}^1$, then the twistor fibration, which is the bundle of compatible almost complex structures, is given by $$\pi: \cp^3 \rightarrow S^4,\quad [x_1:y_1:x_2:y_2]\mapsto [x_1+y_1j:x_2+y_2j].$$ In example \ref{Example: Toric orbifolds} we exhibited $S^4$ as a compact toric 4-manifold corresponding to $\s = \{(1,0),(0,1)\}\subset\Lambda$. With respect to $\s$, the action of $F$ on $S^4$ can be written as $$(\theta_1,\theta_2):[x + yj: 1] \mapsto [e^{\pi i\theta_1}x + e^{\pi i\theta_2}yj: 1].$$ For $q_1, q_2 \in \h$ satisfying $x + yj = q_2^{-1}q_1$ this action can be written as
$$[e^{\pi i(\theta_1 + \theta_2)}q_2^{-1}q_1 e^{\pi i(\theta_1 - \theta_2)}: 1] = [q_1 e^{\pi i(\theta_1 - \theta_2)}: q_2 e^{-\pi i(\theta_1 + \theta_2)}].$$
The induced action of $F$ on $\cp^3$ is thus given by $$[e^{\pi i(\theta_1 - \theta_2)} x_1 : e^{-\pi i(\theta_1 - \theta_2)} y_1 : e^{-\pi i(\theta_1 + \theta_2)} x_2  : e^{\pi i(\theta_1 + \theta_2)}y_2],$$ where $q_1 = x_1 + y_1j$ and $q_2 = x_2 + y_2 j$.
Then, the action of $F$ can be complexified to the following action of $G$ on $\cp^3$  $$(s,t):[x_1:y_1:x_2:y_2] \mapsto [sx_1:ty_1:x_2:sty_2].$$

In Figure \ref{Figure: twistor space} we depict $\cp^3$ using its Delzant polytope \cite{D88} arising from the standard toric perspective. We can label this twistor space within the framework of the previous subsection. The points $z_1^+$, $z_2^+$, $z_1^-$ and $z_2^-$ correspond to $[0:0:1:0]$, $[1:0:0:0]$, $[0:0:0:1]$ and $[0:1:0:0]$, respectively. The edges $L_1$, $L_2$, $C_1^+$, $C_2^+$, $C_1^-$ and $C_2^-$ correspond to the lines $[0:0:x_2:y_2]$, $[x_1:y_1:0:0]$, $[x_1:0:x_2:0]$, $[x_1:0:0:y_2]$, $[0:y_1:0:y_2]$ and $[0:y_1:x_2:0]$, respectively.
\end{TE}

\begin{figure}[htbp]
  \psfrag{C1+}{$C_1^+$}
  \psfrag{C2+}{$C_2^+$}
  \psfrag{C1-}{$C_1^-$}
  \psfrag{C2-}{$C_2^-$}
  \psfrag{L1}{$L_2$}
  \psfrag{L2}{$L_1$}
  \psfrag{z1+}{$z_2^+$}
  \psfrag{z2+}{$z_1^+$}
  \psfrag{z1-}{$z_2^-$}
  \psfrag{z2-}{$z_1^-$}
 \centering \scalebox{1.0}{\includegraphics{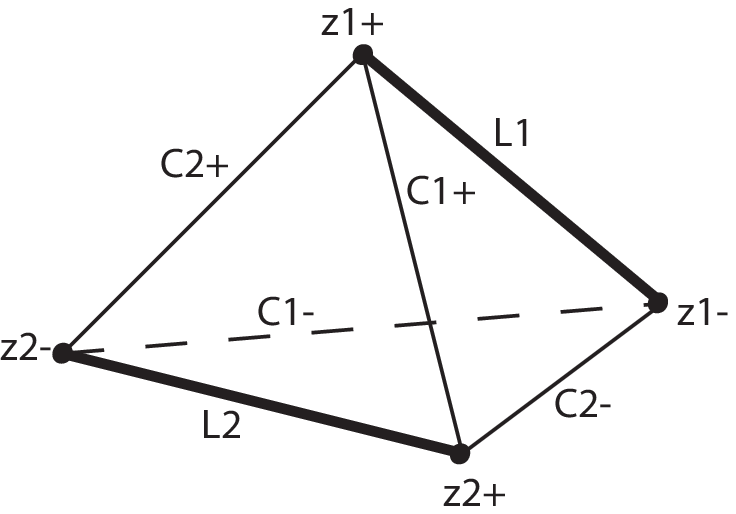}}\label{Figure: twistor space}
  \caption{The Delzant polytope for $\cp^3$}
\end{figure}

\begin{TL}\label{Lemma: First fixed point}
Let $x$ be a smooth point in an anti-self-dual orbifold $U$ with a smooth and effective action of $F$. If $z$ is a $G$-fixed point in the twistor line $L = \pi^{-1}(x)$, then coordinates $(u,v,w)$ can be chosen in a neighbourhood of $z = (0,0,0)$ such that $|u|, |v|, |w| < 1$ and the action of $G$ is given by $$(s,t): (u,v,w)\mapsto(su,tv,stw),$$ for $|s| \leq 1$ and $|t| \leq 1$.
\end{TL}
\begin{proof}
Since $F$ stabilizes $x$, there is an induced action of $F$ on $T_{x} U$, which can be written as $$(\theta_1,\theta_2):(z_1, z_2)\mapsto (e^{2 \pi i\theta_1}z_1, e^{2 \pi i\theta_2}z_2).$$ The conformal structure on $U$ induces a conformal structure on $T_{x} U$, which makes $T_{x} U$ conformally equivalent to $\R^4$.

The twistor space for $T_{x} U$ is given by the normal bundle $N_{L}\cong \mc{O}(1)\oplus\mc{O}(1)$, as explained in Subsection \ref{Subsection: Twistor spaces}. Recalling example \ref{Example: Twistor space}, we can choose coordinates $[x_1: y_1: x_2: y_2]$ (away from $x_2 = y_2 = 0$) on $N_{L}$, where $L$ corresponds to $x_1 = y_1 = 0$ and the action of $G$ is given by
$$(s,t):[x_1:y_1:x_2:y_2] \mapsto [sx_1:ty_1:x_2:sty_2].$$
We will make the choice that $z$ corresponds to $[0:0:1:0] \in N_{L}$ and therefore, the other $G$-fixed point on $L$ corresponds to $[0:0:0:1]$.

If $V$ is an open neighbourhood of $z$ in the twistor space, then there is an isomorphism $T_{z} N_{L} \cong T_{z} V$. Therefore, we can choose coordinates on $T_{z}V$ such that the induced action of $F$ on $T_{z} V$ (with the same $F$-coordinates) is given by $$(s,t):(u,v,w) \mapsto (su, tv, st w),$$
where $z$ corresponds to $(0,0,0)$. There is a $F$-equivariant biholomorphism between a neighbourhood of $0 \in T_{z} V$ and a neighbourhood of $z \in  V$, as noted in Subsection 5.1 in \cite{F00}. This neighbourhood can be chosen to satisfy $|u|< 1, |v|<1$ and $|w|<1$. Furthermore, Lemma 5.1 of \cite{F00} states that the action of $F$ on this neighbourhood can be partially complexified to an action by $(s,t) \in  G$ with $|s| \leq 1$ and $|t| \leq 1$.
\end{proof}

We will denote the orbifold chart about $x_i\in B$ that was constructed in Subsection \ref{Subsection: Boundary orbits and fixed points} by $\varphi_i: \tilde U_i \rightarrow U_i$. We will also denote the lattice with generators $u_{i-1}$ and $u_i$ by $\Lambda_i \subset \Lambda$, and the corresponding 2-torus by $F_i$. So, the orbifold structure group of $x_i$, which we will denote by $\Gamma_i$, is isomorphic to $\Lambda/\Lambda_i$ and the action of $F$ on $U_i$ lifts to an action of $F_i$ on $\tilde U_i$.

Let $\varphi_i^+: \tilde V_i^+ \rightarrow V_i^+$ be the orbifold chart for $z_i^{+} \in \Sigma$ that is induced from $\varphi_i$. The action of $F$ on $V_i^+$ lifts to an effective action of $F_i$ on $\tilde V_i^+$, and we will denote the complexification of $F_i$ by $G_i$. We will denote the preimages of $z_i^\pm$, $L_i$ and $C_i^\pm$ under $\varphi_i^\pm$ by $\tilde z_i^\pm$, $\tilde L_i$ and $\tilde C_i^\pm$, respectively. We will use coordinates $(s, t)$ on $G_i$, where the action of $(s,1)$ and $(1,t)$ stabilize points on $\tilde C_{i-1}^\pm$ and $\tilde C_i^\pm$, respectively.

\begin{TL}\label{Lemma: Fixed point neighbourhood}
Coordinates $(u,v,w)$ can be chosen in a neighbourhood of $\tilde z_i^{\pm} = (0,0,0)$ satisfying $|u|, |v|, |w| < 1$, such that
\begin{equation}\label{Equation: Local model equation}\begin{matrix}\tilde C_{i-1}^{\pm}\cap \tilde V_i^{\pm} = \{u = w = 0\}\\ \tilde C_{i}^{\pm}\cap \tilde V_i^{\pm} = \{v = w =0\} \\ \tilde L_i\cap \tilde V_i^{\pm} = \{u = v = 0\}\end{matrix}.\end{equation} The action of $G_i$ about $\tilde z_i^+$, in the coordinates given above, is $$(s,t): (u,v,w)\mapsto(su,tv,stw),$$ for $|s| \leq 1$ and $|t| \leq 1$. The action of $G_i$ about $\tilde z_i^-$, with respect to the same coordinates on $G_i$, is  $$(s,t): (u,v,w)\mapsto(s^{-1}u,t^{-1}v,s^{-1}t^{-1}w),$$ for $|s| \geq 1$ and $|t| \geq 1$.
\end{TL}
\begin{proof}
Lemma \ref{Lemma: First fixed point} gives a local model for a neighbourhood of $z_1^+$, from which it is easy to deduce the equations labeled (\ref{Equation: Local model equation}) for $z_1^{+}$. It was also noted in Lemma \ref{Lemma: First fixed point} that a model for $z_1^-$ can be constructed by identifying $z_1^-$ with $[0:0:0:1] \in N_{\tilde L_{1}}$ and then following the same procedure. This gives a local model for a neighbourhood of $z_1^-$ satisfying the statement of this lemma. Thus, we have proven the lemma for $z_1^+$ and $z_1^-$.

If we repeat the argument in Lemma \ref{Lemma: First fixed point} for $z_{2}^+$ we are faced with the possibility that $\tilde z_2^{+}$ corresponds to either $[0:0:0:1]$ or $[0:0:1:0]$ in $N_{\tilde L_{2}}$. In the model about $\tilde z_1^+$ we have
$$\tn{lim}_{|t|\rightarrow 0 }\; \varphi_1^+((1,t).\tilde z) = z \in {C_1^+}',$$ for $\tilde z \in \tilde V_1^+$. If we choose $g \in G$ so that $g.z \in V_2^+\cap {C_1^+}'$, then, for $t_0$ sufficiently small,
$$w := g.\varphi_1^+((1,t_0).\tilde z) \in V_2^+.$$
We will write $t^{u_1} := \tn{exp}(u_1 \tn{log}(t))$ and so, the action of $t^{u_1}$ stabilizes $C_1^\pm\cap V_1$ and this lifts to an action of $(1,t)$, which stabilizes $\tilde C_1^\pm\cap\tilde V_1$. Using this notation
$$\tn{lim}_{|t|\rightarrow 0 }\; t^{u_1}.w = \tn{lim}_{|t|\rightarrow 0 }\; g.\varphi_1^+((1,t.t_0).\tilde z) = g.z \in {C_1^+}'.$$
Thus, if we choose $\tilde w \in \tilde V_2^+$ with $\varphi_2^+(\tilde w) = w$, then
$$\tn{lim}_{|s|\rightarrow 0 }\; \varphi_2^+((s,1).\tilde w) = \tn{lim}_{|s|\rightarrow 0 }\; s^{u_1}.\varphi_2^+(\tilde w) = \tn{lim}_{|s|\rightarrow 0 }\; s^{u_1}.w \in {C_1^+}'.$$
However, $$\tn{lim}_{|s|\rightarrow 0 }\; (s,1).\tilde w \in \tilde {C_1^+}',$$ only in the local model where $\tilde z_2^{+}$ corresponds to $[0:0:1:0]\in N_{\tilde L_{2}}$. This proves the lemma for $z_2^+$ and consequently, $z_2^-$. Then, an inductive argument completes the proof.
\end{proof}

\begin{TR}
For each $i=1,\ldots,k$ we have been careful to use the same lift of the $G$-action to the local models about both $z_i^+$ and $z_i^-$. Specifying which lifts we use on the local models is not necessary for establishing that the closure of a $G$-orbit is a complex orbifold; however, it is important for determining the poles and zeros of the meromorphic function that we define in Section \ref{Section: The meromorphic data}.
\end{TR}

\subsection{The local model about non-fixed points in $\Sigma$}\label{subsection: local model non-fixed}
\begin{TL}\label{Lemma: C local model}
If $z\in {C_i^{\pm}}'$, then one can choose coordinates $(u,v,w)$ on an orbifold chart $\varphi: \tilde V \rightarrow V$, satisfying $|u-1|< \epsilon, |v|<1$ and $|w|<1$, such that $\varphi(1,0,0) = z$ and $u=w=0$ defines $\tilde C_i^{\pm}\cap\tilde V$. Furthermore, if $\tilde G$ denotes the lift of $G$ to $\tilde V$, then the action of $\tilde G$ is given by $$(s,t):(u,v,w)\mapsto(su,tv,stw),$$ where $|su-1|< \epsilon$, $|t| \leq 1$, $|st| \leq 1$.
\end{TL}
\begin{proof}
First suppose that $x \in B_i$ is a smooth point. The action of $G$ is transitive on ${C_i^{+}}'$. Therefore, we can chose $g \in G$ so that $g.z\in V_i^+$. Moreover, $g$ can be chosen so that $g.z = (1/2,0,0)$ in the coordinates on $V_i^+$ given by Lemma \ref{Lemma: Fixed point neighbourhood}. There is a neighbourhood of $g.z \in V_i^+$ with coordinates $(u,v,w)$ satisfying $|u-1/2| < \epsilon < 1/2$, $|v|<1$ and $|w|<1$, where the action of $G$ is as described in Lemma \ref{Lemma: Fixed point neighbourhood}. Applying $g^{-1}$ to this coordinate neighbourhood induces a $G$-equivariant biholomorphism to an open neighbourhood of $z$. Thus, with suitable rescaling, we obtain coordinates of the desired form on a neighbourhood of $z$. In this case, it is clear that $v=w=0$ defines $C_i^{\pm}$. This smooth model can then be used as a local model about $\varphi^{-1}(z)$ in an orbifold cover, which completes the proof.

\end{proof}

\begin{TL}\label{Lemma: Twistor line neighbourhood}
If $z \in L_i\backslash z_i^{\pm}$, then one can choose coordinates $(u,v,w)$ on an orbifold chart $\varphi: \tilde V \rightarrow V$, satisfying $|u|< 1, |v|<1$ and $|w-1|<\epsilon$, such that $\varphi(0,0,1) = z$ and $u=v=0$ defines $\tilde L_i \cap \tilde V$. Furthermore, if $\tilde G$ denotes the lift of $G$ to $\tilde V$, then the action of $\tilde G$ is given by $$(s,t): (u,v,w)\mapsto(su,tv,stw),$$ for $|s| \leq 1$, $|t| \leq 1$ and $|stw-1| < \epsilon$.
\end{TL}
\begin{proof}
This lemma can be proven in much the same way as lemma \ref{Lemma: C local model}.
\end{proof}

\section{The orbit closures}\label{Section: Analyticity of orbit closures}
\subsection{The local foliation}\label{Subsection: Local analyticity}
Recall that $C := \cup \, C_i^{\pm}$, $\mc{L} := \cup\, L_i$ and $\Sigma = C \cup \mc{L}$. We will continue to use the same notation as Section \ref{Section: Local model} and, in addition, if $V$ is a neighbourhood in $Z$ of a point in $\Sigma$, then we will write $V' := V - \Sigma$ and $\tilde V' = \varphi^{-1}(V')$. Throughout this section we will refer to the coordinate neighbourhoods of points in $\Sigma$, defined in Subsections \ref{Subsection: Local Model} and \ref{subsection: local model non-fixed}, as \emph{admissible neighbourhoods}.

For an admissible neighbourhood $\varphi: \tilde V \rightarrow V$ about  $z \in {C_i^{\pm}}'$ we define
\begin{equation}\label{Equation: Parameterizing orbits} \tilde W_{ab}':= \{(u,v,w)\in \tilde V' : auv = bw \},\;\textnormal{ for }\;(a,b)\neq(0,0).\end{equation}
We will denote the image of $\tilde W_{ab}$ under $\varphi$ by $W_{ab}$. Since $\Gamma^i \subset \tilde F$, the $W_{ab}$ are connected closed submanifolds in $V'$, effectively parameterized by $[a:b]\in \cp^1$, each of which is contained in a single $G$-orbit. The closure of $W_{ab}'$ in $V$ is a suborbifold containing $C_i^{\pm} \cap V$.

Let $\varphi: \tilde V \rightarrow V$ be an admissible neighbourhood of $z_i^{\pm}$. The decomposition of $V$ into $G$-orbits is much the same, except $W_{10}$ splits into the closure of two orbits, which intersect along $L_i$. We define $\tilde W_{10}' = \tilde W'_{i-1} \bigsqcup \tilde W'_i$, where $$\tilde W_{i-1} = \{(u,v,w) \in \tilde V': u = 0\}\quad\textnormal{and}\quad \tilde W_{i} = \{(u,v,w) \in \tilde V': v = 0\}.$$ Otherwise, we define $\tilde W_{ab}'$ as before. We will denote the image of $\tilde W_{i-1}$ and $\tilde W_{i}$ under $\varphi$ by $W_{i-1}$ and $W_{i}$, respectively, and we will refer to $W_{i-1}$ and $W_{i}$ as \emph{special leaves}. We have chosen this notation because, according to lemma \ref{Lemma: Fixed point neighbourhood}, the closure of $W_{i-1}$ in the admissible neighbourhood $V_i^\pm$ contains $V_i^\pm\cap C^{\pm}_{i-1}$, while the closure of $W_{i}$ contains $V_i^\pm\cap C^{\pm}_{i}$.

The following lemma was proven in Section 5 of \cite{F00} in the smooth case.

\begin{TL}\label{Lemma: Foliation}
If $W$ is an open neighbourhood of $C$ covered by admissible neighbourhoods, then every point in $W' : = W - \Sigma$ has trivial $G$-stabilizer.
\end{TL}
\begin{proof}
Let $V$ be an admissible neighbourhood and suppose that $g\in G$ stabilizes $z\in V'$. Note that we can always choose $V$ to be an admissible neighbourhood of a  \emph{non-fixed} point in $C$ containing $z$. Therefore, without loss of generality, we can assume that $V'$ is a neighbourhood of a point in ${C_i^\pm}'$. Thus, $x\in W_{ab}$ for some $[a:b]$ and we have $$g(1,t)z = (1,t)gz = (1,t)z,$$ where $(1,t)$ is the generator of $G^i$. We can use the local model given in Lemma \ref{Lemma: C local model} to show that, as $t\rightarrow 0$, $$gz' = z'$$ for some $z'\in {C_i^\pm}'$. Only elements of $G^i \subset G$ fix $z$ (by Lemma \ref{Lemma: Defining C}) and so, $g = (1,t_0)\in G^i$. By Lemma \ref{Lemma: C local model}, if $g = (1,t_0)\in G^i$ with $|t_0| \leq 1$, then $g = (1,1)$. If $|t_0|>1$, then consider $g^{-1}$.
\end{proof}

A smooth and effective group action on an orbifold defines a foliation provided it has only finite stabilizers \cite{L74}. Therefore, Lemma \ref{Lemma: Foliation} proves that the action of $G$ on $Z$ defines a foliation of $W'$ by complex 2-orbifolds. In the next subsection we prove that the closure of a $G$-orbit intersecting $W$ contains only a single leaf of this foliation, from which it follows that the orbit closure is a complex orbifold.

\begin{TD}\label{Definition: Special orbits}
Let $W$ be the neighbourhood of $C$ defined in Lemma \ref{Lemma: Foliation}. Recall that in an admissible neighbourhood of $z_i^{\pm}$ we have two special leaves $W_{i-1}$ and $W_{i}$. We will refer to a $G$-orbit containing a special leaf for some $z_i^{\pm}$ as \emph{special}, otherwise we will refer to a $G$-orbit intersecting $W'$ as \emph{non-special}. \end{TD}
\begin{TD}\label{Definition: B-orbits}
By Lemma \ref{Lemma: Constructing C}, we can choose a non-fixed point in the twistor line $L_x$ over the non-fixed point $x \in B_i$. The closure of the $G^i$-orbit of a non-fixed point in $L_x$ is $L_x$. Therefore, $L_x$ is contained in the closure of a $G$-orbit, which we will denote by $O_x$. We will refer to such orbits as \emph{$B$-orbits}, since they correspond to boundary orbits in $M$.
\end{TD}

\begin{TL}\label{Lemma: Infinite B-orbits}
There exists $x\in B$ such that the $B$-orbit $O_x$ is non-special.
\end{TL}
\begin{proof}
Suppose that there are only finitely many $B$-orbits, then the union of all $B$-orbits has 2 complex dimensions. From the definition of the $B$-orbits, the union of all $B$-orbits coincides with the $G$-orbit of the set $\pi^{-1}(B)\backslash \Sigma$, where $\Sigma$ was defined in Subsection \ref{Subsection: Describing a singular set}. Since $\pi^{-1}(B)\backslash \Sigma$ is a 4 real dimensional set and the union of $B$-orbits is 2 complex dimensional, $\pi^{-1}(B)\backslash \Sigma$ must be $G$-invariant. Moreover, the action of $G$ on $\pi^{-1}(B)\backslash\Sigma$ is locally free (Lemma \ref{Lemma: Constructing C}) and so, it must be 2 complex dimensional. However, $\pi^{-1}(B)\backslash\Sigma$ is not a complex submanifold of $Z$: if it were every lift of $B$ to $Z$ would be holomorphic. This gives a contradiction.

While there are infinitely many $B$-orbits, it follows from definition \ref{Definition: Special orbits} that there are finitely many special orbits. Thus, we can always choose $x \in B$ so that $O_x$ is non-special.
\end{proof}

\begin{TE}
\textnormal{In Figure \ref{Figure: twistor space} we exhibited $\cp^3$ as the twistor space of the toric manifold $S^4$. The special orbits fit into this example as the faces of the tetrahedron representing $\cp^3$. More explicitly, the two special orbits intersecting along $L_1$ correspond to  $$\{[0:s^{-1}t:s^{-1}:t]\} \quad\tn{and}\quad \{[st^{-1}:0:t^{-1}:s]\},$$ while the two special orbits intersecting along $L_2$ correspond to $$\{[s:t:0:st]\} \quad\tn{and}\quad\{[t^{-1}:s^{-1}:s^{-1}t^{-1}:0]\}.$$ An example of a $B$-orbit is
$$\{[as:-\bar{a}t:1:st]\},$$ which contains the twistor line $$\{[az: -\bar{a}:1:z]\}$$ for $a \in \C^*$.
}
\end{TE}

\subsection{The orbit closures as complex orbifolds}\label{Subsection: analyticity of orbit closures}
In this subsection we prove several lemmas, which are summarized in the next proposition.
\begin{TP}\label{Proposition: Analytic orbits}
Let $O$ be a $G$-orbit intersecting $W'$. Then the closure of $O$ is a complex orbifold and $G$ acts freely on $O$. Furthermore, $C$ is connected and, if $O$ is non-special, $$\overline O = O \cup C.$$
\end{TP}

\begin{TL}\label{Lemma: No stabilizer}
If $O\cap W' \neq \emptyset$, then $G$ acts freely on $O$. In particular, $O\cong \C^*\times\C^*$.
\end{TL}
\begin{proof}
For $z\in O$ there exists $g\in G$ such that $g.z \in O\cap W'$. We know from Lemma \ref{Lemma: Foliation} that $g.z$ has trivial $G$-stabilizer. Therefore, $z$ has trivial stabilizer.
\end{proof}

\begin{TL}\label{Lemma: C connected}
The set $C$ is connected and any non-special orbit $O$ satisfies $$\overline{O} = O\cup C.$$
\end{TL}
\begin{proof}
By Definition \ref{Definition: Special orbits}, a non-special orbit intersects $W'$. First we prove that if $\overline{O}$ intersects a connected component of $C$, then it must contain that component. Let $C_0$ be a connected component of $C$ chosen so that $O$ intersects an admissible neighbourhood of a point in $C_0$. Let $V$ be one such admissible neighbourhood. Without loss of generality we can assume that $V$ is an admissible neighbourhood of a point in ${C_i^{\pm}}'$, for some $i=1,\ldots,k$. By the local model about a non-fixed point, given in Lemma \ref{Lemma: C local model}, each leaf in $V'$ has closure containing $V\cap C_0$. Thus, the closure of $O$ contains $V \cap C_i^{\pm}$ and therefore, $\overline{O}$ contains $C_i^{\pm}$.

Suppose $i\neq k$. As $O$ contains $C_i^\pm$, it intersects a neighbourhood of $z_{i+1}^\pm$. Since $O$ is non-special, Lemma \ref{Lemma: Fixed point neighbourhood} shows that $V_i^\pm \cap C_0 \subset \overline{O}$. Consequently, $O$ has non-empty intersection with an admissible neighbourhood of a point in $C_{i+1}^{\pm}$ and we can proceed inductively until $i=k$. When $i=k$, $O$ intersects an admissible neighbourhood of $z_1^+$ or $z_1^-$. In either case the inductive step given above will hold, and we can conclude that $C_0\subset \overline{O}$.


Let $W_0$ be an open neighbourhood of $C_0$ covered by admissible neighbourhoods. We want to see that the boundary of $\overline{O}$, which is connected, is precisely $C_0$. Since $C_0 \subset \overline{O}$, it suffices to show that $W_0 - C_0$ contains no points in the closure of $O$. This can be done by showing that no two leaves in an admissible neighbourhood are contained in the same $G$-orbit. Provided at least one of the leaves is non-special, the proof of this is more or less the same as the proof of Lemma \ref{Lemma: Foliation}, which shows that each leaf has trivial stabilizers, (or one could refer to Lemma 5.5 of \cite{F00}). In this case we can assume that both leaves are non-special, since $O$ is non-special.

Now suppose that $C$ is not connected. Then, by the description in Lemma \ref{Lemma: Defining C}, $C$ consists of two components $C_0$ and $C_1$, which are interchanged by $\gamma$. If we choose a non-fixed point $x\in B$, then the corresponding twistor line $L_x$ intersects both $C_0$ and $C_1$. By Lemma \ref{Lemma: Infinite B-orbits}, we can choose $x$ so that the $B$-orbit $O_x$ is non-special. Therefore, $\overline{O}_x$ intersects both $C_0$ and $C_1$ and it follows that $\overline{O}_x$ must contain both $C_0$ and $C_1$. However, Lemma \ref{Lemma: No stabilizer} proves that $\overline{O}_x$ has a connected boundary.
\end{proof}

\begin{TL}\label{Lemma: analytic closure}
If $O$ is a $G$-orbit with $O\cap W' \neq \emptyset$, then the closure of $O$ is a complex orbifold.
\end{TL}
\begin{proof}
As shown in Lemma \ref{Lemma: C connected} the closure of a non-special orbit intersects $W$ in a suborbifold that comprises a single leaf of the foliation and $C$. The union of this suborbifold in $W$ and the remainder of the $G$-orbit in $Z$ is thus a complex orbifold. Therefore, it remains to show that a special orbit $O$ is a complex orbifold. By Definition \ref{Definition: Special orbits}, $O$ contains at least one of the special leaves $W_{i-1}$ or $W_{i}$ in the admissible neighbourhood $V_i^\pm$. The closure of each of these leaves contains $L_i\cap V_i^{\pm}$ and so, $L_i\cap V_i^{\pm}\subset \overline{O}\cap V_i^{\pm}$. Therefore, $L_i \subset \overline{O}$. Consequently, we can use an argument similar to the non-special case (in Lemma \ref{Lemma: C connected}) to show that the boundary of $O$ is connected and it must contain a cycle of curves in $\Sigma$.

To prove that $O$ is a complex orbifold it suffices to show that, where $O$ intersects an admissible neighbourhood of a point in $\Sigma$, the intersection is a single leaf and the closure of this leaf is a complex orbifold. Then we can proceed as with the non-special orbits. This is dealt with in Lemma \ref{Lemma: C connected} for points in admissible neighbourhoods of ${C_i^\pm}'$. Therefore, it remains to show that $O$ contains a single leaf in an admissible neighbourhood of a point in $L_i$. So, let $\varphi: \tilde V \rightarrow V$ be an admissible neighbourhood of a point on $L_i$. In the coordinates $(u,v,w)$ on $\tilde V$, which were defined in Lemmas \ref{Lemma: Fixed point neighbourhood} and \ref{Lemma: Twistor line neighbourhood}, $\varphi^{-1}(\overline{O}\cap V)$ contains at least one of the sets $\{u=0\}$ or $\{v=0\}$. We will assume that $\varphi^{-1}(\overline{O}\cap V)$ contains $\{u=0\}$, as the argument proceeds similarly for the other choice.

Now note that $\varphi(\{u=0\})$ is a complex suborbifold in $V$. Therefore, to complete the proof we need to show that $\{u=0\}\cap V'$ is the only $G$-orbit in $\tilde V'$ contained in $O$. We can choose $g\in G$ so that $$g.V \subset V_i^{\pm}.$$ Therefore, it suffices to show that $V_i^{\pm} \cap O$ contains a single local leaf. We have already noted in Lemma \ref{Lemma: C connected} that no two local leaves in $V_i^\pm$ are contained in the same $G$-orbit when at least one of the leaves is non-special. Now suppose two special leaves are contained in the same $G$-orbit or equivalently, there exist $h \in G$ such that $$hx = y,$$ where $x \in W_i$ and $y \in W_{i-1}$. If we multiply by the generator of the $G^{i-1}$-action $(s,1)$, then we obtain
$$h(s,1)x = (s,1)y.$$
If we take the limit as $s\rightarrow 0$, Lemma \ref{Lemma: Fixed point neighbourhood} shows that $$h.x' = z_i^{\pm},$$
where $x'$ is a non-fixed point in $C_{i-1}^{\pm}$. This gives a contradiction.
\end{proof}

\section{The involution}\label{Section: The Involution}
\subsection{Principal lines}\label{Subsection: Intersecting orbits analytic}
We will denote the four complex dimensional space of twistor lines in $Z$ by $\mC$, and the subspace of real twistor lines in $\mC$ by $\m$ (cf. Subsection \ref{Subsection: Twistor spaces}). Points in $\m$ are fibres of the twistor projection $\pi: Z \rightarrow M$ and so, there is a natural $F$-equivariant diffeomorphism between $\m$ and $M$. We define a \emph{principal line} to be a member of the set
$$\mCo := \{L \in \mC : \tn{The action of $G$ at every point in }L\tn{ has a finite stabilizer}\}$$
and a \emph{real principal line} to be member of the set
$$\mo := \mCo \cap \m.$$ So, $\mo$ consists of the twistor lines over points in either principal orbits or exceptional orbits in $M$.
The aim of this subsection is to determine the intersection number of the $G$-orbits with a principal line. To do this we must first establish that the closure of each of the $G$-orbits intersecting a principal line is a complex orbifold.

\begin{TL}\label{Lemma: Intersecting orbits}
If $L \in \mCo$, every $G$-orbit intersecting $L$ has closure intersecting $C$.
\end{TL}
\begin{proof}
To prove this lemma we show that $$L_C := \{z\in L : \overline{G.z} \cap C \neq \emptyset\},$$ is non-empty, open and closed in $L$.

In the notation of Section \ref{Section: Analyticity of orbit closures}, let $W$ be an open neighbourhood of $C$ covered by admissible neighbourhoods and let $W' = W - \Sigma$. In Proposition \ref{Proposition: Analytic orbits} we showed that the $G$-orbits intersecting $W'$ are complex orbifolds; thus, their closures have a well-defined intersection number with every twistor line. This number is non-zero, since such orbits have non-empty intersection (along $C$) with the twistor lines in $\pi^{-1}(B)$. Consequently, $L_C$ is non-empty.

To show that $L_C$ is open, let $z\in L_C$ and choose $g\in G$ such that $g.z$ is contained in $W'$. We can choose an open ball $U$ about $g.z$ in $W'$ and define a diffeomorphism $$g^{-1}: U \rightarrow g^{-1} U;\,\, u \mapsto g^{-1} u.$$ Therefore, $g^{-1} U$ is an open ball in $Z$ about $z$, and $(g^{-1} U) \cap L$ is an open set in $L$ containing $z$. Moreover, the $G$-orbit of each point in $(g^{-1} U) \cap L$ intersects the admissible neighbourhood $W'$ and so, the closure of these $G$-orbits intersect $C$. Hence, $L_C$ is open.

To show that $L_C$ is closed, let $(z_i)_{i\geq1}$ be a sequence in $L_C$ converging to $z_\infty \in L$. We can assume, without loss of generality, that the $z_i$ and $z_\infty$ are not contained in special orbits, since we know such orbits have closure intersecting $C$. If we choose an admissible neighbourhood $V$ of a non-fixed point in $C$, then each $z_i$ corresponds to a $G$-orbit intersecting $V'$. Recall from Subsection \ref{Subsection: Local analyticity} that the $G$-orbits intersect $V'$ in leaves $W_{ab}'$, which are parameterized by $\cp^1$. Therefore, away from the special orbits there is a well-defined map from $L_C$ to $\cp^1$. The sequence $(z_i)_{i\geq 1}$ is mapped to $([a_i\!:\!b_i])_{i \geq 1}$ in $\cp^1$, and by choosing a convergent subsequence we can assume that $[a_i\!:\!b_i] \rightarrow [a_\infty\!:b_\infty]$.

Now choose $v_\infty \in W_{a_\infty b_\infty}'$. From the definition of $W_{ab}'$ in (\ref{Equation: Parameterizing orbits}), it is clear that there exists a sequence $(v_i)_{i \geq 1} \subset V'$ such that $v_i \in W_{a_i b_i}$ and $v_i \rightarrow v_\infty$ as $i \rightarrow \infty$. So, if $U \subset V'$ is an open ball about $v_\infty$, then $U$ contains infinitely many $v_i$.

The closure of the $G$-orbit containing $v_\infty$ is a complex orbifold and so, it has a positive intersection number with $L$. Therefore, we can choose $g_1,\ldots,g_n \in G$, so that $$g_j.v_\infty\in L, \textnormal{ for } j=1,\ldots,n.$$ By the argument used to prove openness, $g_j.U$ intersects $L$ in an open neighbourhood containing $g_j.v_\infty$. As $g_1, \ldots, g_n$ are fixed, we can shrink $U$ to obtain an arbitrarily small open neighbourhood about each $g_j.v_\infty$ in $L$. Moreover, for some $m \in \{1,\ldots, n\}$, this neighbourhood contains infinitely many $z_i$, since $U$ contains infinitely many $v_i$. It follows that a subsequence of $z_i$ converges to $g_m.v_\infty$. So, by uniqueness of limits, $g_m.v_\infty = z_\infty$ and therefore, $z_\infty \in L_C$.
\end{proof}

\begin{TR}
By Lemma \ref{Lemma: No stabilizer} the action of $G$ on every $G$-orbit intersecting $C$ is free. Therefore, the action of $G$ is free about $L$, by the previous lemma. Although we have not explicitly mentioned the exceptional orbits, defined in Subsection \ref{Subsection: Compact 4-orbifolds}, their existence has remained a possibility, which we can now rule out. An exceptional orbit is an $F$-orbit in $M$ with a finite stabilizer subgroup. In the twistor space this stabilizer subgroup acts as biholomorphisms preserving the principal lines corresponding to the exceptional orbit. Each of these biholomorphisms has fixed points and so, we can find points in these ``exceptional" principal lines that are fixed by an element of the $F$-action. This gives a contradiction. Moreover, this rules out orbifold points on principal lines.
\end{TR}

\subsection{The intersection number}\label{Subsection: Intersction number}
We can now prove that ``a generic $G$-orbit intersects a generic twistor line twice".
\begin{TL}\label{Lemma: Intersection number}
Every non-special $G$-orbit has intersection number two with every principal line.
\end{TL}
\begin{proof}
Let $O$ be a non-special $G$-orbit. We can choose a non-fixed point $x\in B$ so that $L := \pi^{-1}(x)$ is not contained in $O$, since $L$ is contained in the closure of a single $B$-orbit. Recall from Lemma \ref{Lemma: Constructing C} that $L$ intersects $C$ in precisely two distinct points: $z$ and $\gamma(z)$. It follows from Lemma \ref{Lemma: C connected} that $z$ and $\gamma(z)$ are contained in the closure of $O$. What is more, these intersections are transverse, as can be seen from the model for an admissible neighbourhood in Lemma \ref{Lemma: C local model}. So, $\overline{O}$ has intersection number at least 2 with $L$.

By our choice of $L$, $\overline{O}$ does not intersect $L$ away from $z$ and $\gamma(z)$. Therefore, $\overline{O}$ intersects $L$ at precisely two points, each with multiplicity one. Since $\overline{O}$ is a complex orbifold, it has a well-defined intersection number with any twistor line. Moreover, the twistor lines are homologous  and so, for any $L \in \mC$, $\overline{O} \cdot L = 2.$ Thus, we can conclude that $$O \cdot L = 2,$$ when $L \in \mCo$, since a principal line does not intersect $C$.
\end{proof}

\begin{TL}\label{Lemma: Special orbit intersection}
Every special $G$-orbit has intersection number one with every principal line. Furthermore, the two special orbits with closures intersecting along $L_i$ are interchanged by the real structure.
\end{TL}
\begin{proof}
Let $O$ be a special orbit, so by definition, $O$ contains a special leaf in the admissible neighbourhood $V_i^\pm$. Since $\overline{O}$ contains $L_i$, it must also intersect $V_i^\mp$. Therefore, we will suppose that $O$ intersects $V_i^+$ without loss of generality. In Lemma \ref{Lemma: Intersecting orbits} we proved that $O$ can only contain one leaf in the foliation of $V_i^+$ by $G$. So, suppose that $O$ contains the special leaf $W'_i$, which has closure intersecting $C_i^{+}$. (The argument proceeds similarly if $O$ contains $W'_{i-1}$.) Since $O$ contains no other leaves in $V_i^+$, the closure of $O$ does not contain $C_{i-1}^+$.

Similarly to the non-special case in Lemma \ref{Lemma: Intersection number}, we can choose $L \in \m$ containing non-fixed points $z \in {C_{i-1}^+}$ and $\gamma(z) \in {C_{i-1}^-}$. Since $\overline{O}$ does not contain $C_{i-1}^+$, it only intersects $L$ at $\gamma(z)$. As noted in Lemma \ref{Lemma: Intersection number}, this intersection is transverse. Then, proceeding by the same argument as in Lemma \ref{Lemma: Intersection number}, we can conclude that special orbits have intersection number one with principal lines.

The $G$-orbit $\gamma(O)$ has closure intersecting $L_i$ and therefore, $\gamma(O)$ is a special orbit containing one of the local leaves in $V_i^+$. We have established that $O$ contains $\gamma(z) \in C_{i-1}^-$ and so, $\gamma(O)$ contains $z\in C_{i-1}^+$. Consequently, $\gamma(O)$ is the special orbit containing $W_{i-1}'$. Thus, $\gamma$ interchanges the two special orbits intersecting along $L_i$.
\end{proof}

\begin{TD}
We will denote the special orbit that contains the special leaf $W'_{i}\subset V_i^+$ by $O_i^+$, and the special orbit that contains the special leaf $W'_{i-1}\subset V_i^+$  by $O_i^-$. By the proof of the preceding lemma, $O_i^+$ contains $W'_{i-1}$ in $V_i^-$ and $O_i^-$ contains $W'_{i}$ in $V_i^-$. Thus, there are at most $2k$ special orbits: $$O_1^+, \ldots, O_k^+,O_1^-, \ldots, O_k^-.$$
\end{TD}

In the next lemma we assume that $\overline{O}_i^+$ only contains a single twistor line. This lemma will be used to prove Lemma \ref{Lemma: Special point ordering}, once we have established this assumption in Lemma \ref{Lemma: Special points}.

\begin{TL}\label{Lemma: Special orbit decomposition}
If $L_i$ is the only twistor line contained in $\overline{O}_i^+$, then $$\overline{O}_i^+ = \bigcup_{l = i}^k C_{l}^+ \cup \bigcup_{l = 1}^{i-1} C_{l}^- \cup L_i \cup O_i^+.$$
\end{TL}
\begin{proof}
From the definition of $O_i^+$ given above and the definition of a special leaf in Subsection \ref{Subsection: Local analyticity}, it follows that $\overline{O}_i^+$ contains $C_i^+$. The only twistor line contained in $\overline{O}_i^+$ is $L_i$. Thus, $\overline{O}_i^+$ does not contain a special leaf in the admissible neighbourhoods $V_j^\pm$, for $j \neq i$. Using the same argument as in the proof of Lemma \ref{Lemma: C connected}, we can conclude that $$\bigcup_{l = i}^k C_{l}^+ \cup \bigcup_{l = 1}^{i-1} C_{l}^- \cup L_i \subset \overline{O}_i^+.$$ Then, by the proof of Lemma \ref{Lemma: analytic closure}, no other points are contained in the closure of $O_i^+$.
\end{proof}

\begin{TR}
The results of this subsection show that the closure of the $G$-orbit of a complex structure can be thought of as a double cover of $M$. The covers corresponding to $B$-orbits are branched at $F$-orbits of twistor lines in $\pi^{-1}(B)$ (by Definition \ref{Definition: B-orbits}), while the remaining non-special orbits correspond to covers that are branched over principal orbits in $M$, as we see in the next subsection. The closure of each pair of special orbits $\overline{O}_i^+\cup\overline{O}_i^+$ is also a double cover branched at $L_i$.
\end{TR}

\begin{figure}[htbp]
  \psfrag{a}{$\infty$}
  \psfrag{b}{0}
  \psfrag{c}{$C$}
  \psfrag{d}{$C$}
  \psfrag{z}{$z$}
  \psfrag{w}{$-z$}
  \psfrag{L}{$L_i$}
  \psfrag{O}{$O_i^+$}
  \psfrag{P}{$O_i^-$}
  \centering \scalebox{0.8}{\includegraphics{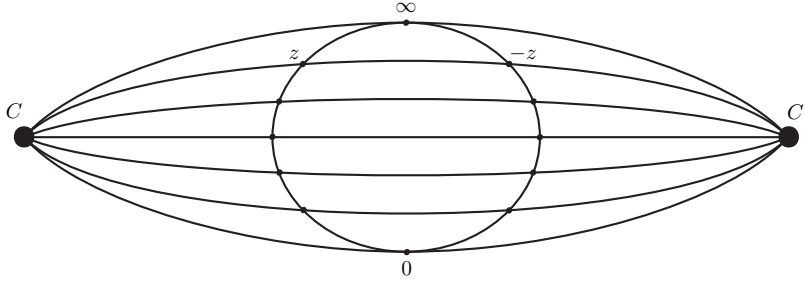}}
  \caption{$G$-orbits intersecting a principal line}\label{Figure: Involution}
\end{figure}

\subsection{An involution on a principal line}\label{Subsection: involution}
Figure \ref{Figure: Involution} summarizes the results from the previous subsection: the figure depicts non-special $G$-orbits intersecting a principal line $L$ twice and the pair of special orbits $O_i^\pm$ intersecting once. The results of Lemmas \ref{Lemma: Intersecting orbits} and \ref{Lemma: Special orbit intersection} are also included in the figure: the closure of every $G$-orbit through $L$ intersects $C$, while the closure of the special orbits also intersect along $L_i$.

We now describe how the information presented in this picture determines an involution. We will refer to a point in $L$ on a special orbit as a \emph{special point} and we define $L'$ to be the complement of the special points in $L$. It follows from Lemma \ref{Lemma: Intersection number} that there is a non-trivial holomorphic involution on $L'$, which we will denote by $\tau$, defined so that $z$ and $\tau(z)$ belong to the same $G$-orbit. In a punctured neighbourhood of a special point $\tau$ is a biholomorphism; thus, $\tau$ extends to a holomorphic involution on $L$ and (with an appropriate choice of coordinate on $L$) can be written as $$\tau: z \mapsto -z.$$ At the fixed points of $\tau$ the action of $G$ is tangential to $L$. Therefore, $0$ and $\infty \in L'$, since the special orbits are not tangential to $L$.

When $L \in \mo$, $0$ and $\infty$ correspond to the two complex structures on the tangent plane of $\pi(L)$ that preserve the tangent plane to the $F$-orbits and so, $0 = \gamma(\infty)$. It follows that the real structure is given by $z \mapsto c.\bar{z}^{-1},$ where $c$ is a negative real number. If the coordinate is multiplied by an appropriate factor, then the restriction of the real structure to $L$ can be written as
\begin{equation}\label{Equation: Real Structure}\gamma: z \mapsto -\frac{1}{\bar{z}}.\end{equation}

\begin{TL}\label{Lemma: Involution}
If $L \in \mCo$, then $L$ is smooth and the action of $G$ is free at every point. Furthermore, an affine coordinate $z$ can be chosen on $L$ such that:
\begin{enumerate}
\item{the action of $G$ is tangential to $L$ at $0, \infty$;}
\item{the involution $\tau$ induced by the action of $G$ is given by $z \mapsto -z$;}
\item{when $L \in \mo$, the real structure $\gamma$ is given by $z \mapsto -\bar{z}^{-1}$.}
\end{enumerate}
\end{TL}

\begin{TR}\label{Remark: Coordinates}
From now on we will use a coordinate on $L$ satisfying the previous lemma. However, there is not a unique coordinate preserving $\tau$ and the pair $0, \infty$. There are two types of coordinate transformation that are permissible: firstly the map $$z \mapsto z^{-1},$$ which interchanges $0$ and $\infty$; and secondly, the map $z \mapsto \lambda.z$ for $\lambda \in \C^*$, which fixes $0$ and $\infty$. When $L \in \mo$, we require $|\lambda| = 1$ to preserve equation (\ref{Equation: Real Structure}).
\end{TR}

\begin{TL}\label{Lemma: B-orbits}
The $B$-orbits are non-special and intersect $L \in \mo$ at points in $$\{z \in L: |z| = 1\}.$$
\end{TL}
\begin{proof}
Recall from Definition \ref{Definition: B-orbits} that a $B$-orbit $O$ has closure containing the twistor line $L := \pi^{-1}(x)$, for some $x \in B$. Since $\overline{O}$ is the only $G$-orbit containing $L$ and $\gamma(L) = L$, it follows that $\gamma(O) = O$. Therefore, $O$ is non-special, as special orbits are interchanged by the real structure. Consequently, $O$ must intersect $L$ at the two points $z$ and $-z$. Then, since $\gamma(O) = O$, it follows that $-z = -\bar{z}^{-1}$.
\end{proof}

\begin{figure}[htbp]
  \psfrag{A}{$\psi(z)$}
  \psfrag{B}{$0$}
  \psfrag{C}{$\infty$}
  \psfrag{D}{$z$}
  \psfrag{E}{$-z$}
  \psfrag{F}{$-z_3$}
  \psfrag{G}{$-z_2$}
  \psfrag{H}{$-z_1$}
  \psfrag{I}{$z_1$}
  \psfrag{J}{$z_2$}
  \psfrag{K}{$z_3$}
  \centering \scalebox{1.0}{\includegraphics{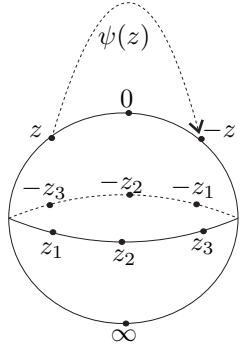}}
  \caption{A real principal line}\label{Figure: A principal line}
\end{figure}

\section{The meromorphic function}\label{Section: The meromorphic data}
\subsection{A meromorphic function on principal lines} \label{Subsection: Meromorphic data}
Figure \ref{Figure: A principal line} depicts a real principal line (when $k = 3$). The tangency points are marked as $0$ and $\infty$, the intersection points with $B$-orbits are located along the equator, as are the special points. In the figure we also represent $\psi(z)$: the action of $G$ required to map $z$ to $-z$. In this subsection we see that $\psi$ can be extended to a meromorphic function. The importance of $\psi$ becomes particularly clear in Section \ref{Section: The conformal structure}: we show that it allows us to reconstruct explicitly the twistor space of the principal lines and consequently, recover the conformal structure over the principal orbits in $M$.

Let $L \in \mCo$ and recall that $L'$ denotes the complement of the special points. There is a well-defined \emph{holomorphic map} $\psi:L' \rightarrow G$ satisfying $$\psi(z).z = -z,$$ which cannot be extended to a $G$-valued holomorphic function over special points, since special orbits only intersect $L$ with multiplicity one. On $L'$ we can write $\psi$ as $(\psi_1, \psi_2)$ with the respect to the coordinates we use to identify $G$ with $\C^*\times\C^*$. In Proposition \ref{Proposition: meromorphic psi} we extend $\psi$ to a \emph{meromorphic} function over $L$; more precisely, $\psi$ extends to a holomorphic map from $L$ to $\cp^1\times \cp^1$. We will say that $\psi$ has a pole of order $v = (a,b)$ at a special point if $\psi_1$ extends with a pole of order $a$ and $\psi_2$ extends with a pole of order $b$. Note that by a pole of order $-a$ for $a > 0$, we mean a zero of order $a$.

\begin{TD}\label{Definition: Combinatorial data}
Recall from Section \ref{Section: Compact Toric 4-Orbifolds} that $B_i$ is labeled by $u_i \in \s \subset \Lambda$, where $\Lambda$ is the lattice in $\mf{f}$ defining $F$. We set $u_0 = -u_k$ and $v_i = u_i - u_{i-1}$ for $i=1,\ldots, k$. This sets up a bijective correspondence between $\s$ and the ordered set $$\mc{T} : = \{v_1,\ldots,v_k\}\subset \Lambda.$$
\end{TD}

In the notation of Section \ref{Section: Local model}, let $V_i^+\subset Z$ be an admissible neighbourhood of $z_i^+ \in C$. Recall from Lemma \ref{Lemma: Fixed point neighbourhood} that there is an orbifold chart $$\varphi_i^+:\tilde V_i^+ \rightarrow V_i^+ \cong \tilde V_i^+/\Gamma_i$$ such that: $\tilde V_i^+$ has coordinates $(u,v,w)$; the action of $G$ lifts to the action of $G_i$ \begin{equation}\label{Equation: G-action}(s,t):(u,v,w)\mapsto (s u, t v,s t w);\end{equation} and the respective generators $s^{(1,0)}$ and $t^{(0,1)}$ of $G_i$ are mapped to $s^{u_{i-1}}$ and $t^{u_{i}}$ in $G$. (Here we use the same notation as Lemma \ref{Lemma: Fixed point neighbourhood}, where $t^{u} := \tn{exp}(u \tn{log}(t))$.)

In these coordinates, $v=0$ and $u=0$ are the preimages of the special leaves $W'_{i}$ and $W'_{i-1}$, respectively, and $\varphi^{-1}(z_i^\pm) = (0,0,0)$. Recall that $O_i^+$ denotes the special orbit that contains the special leaf $W'_{i}\subset V_i^+$, while $O_i^-$ denotes the special orbit that contains the special leaf $W'_{i-1}\subset V_i^+$. By Lemma \ref{Lemma: Intersection number}, $O_i^+$ intersects $L$ at a single point, which we label $z_i$.

\begin{figure}[htbp]
$$\quad \tilde\psi(v) = \tilde f(v)v^{(1,-1)} \quad\quad\quad\quad\quad\quad\quad\quad\quad\quad\quad\quad\quad \quad\psi(v) = f(v).v^{u_{i-1}-u_i}$$
\begin{diagram}
\tilde z_i \in \tilde U_+ \subset \tilde V_i^+          &  \rTo^{\quad\varphi_i^+\quad}  &  g.z_i \in g.U_+ \subset V_i^+  &  \rTo^{\quad g^{-1}\quad}  &  z_i \in U_+ \subset L     \\
\dTo^{v \mapsto \tilde\psi(v).v}                        &                      &                                 &                 &  \dTo_{v \mapsto \psi(v).v}  \\
-\tilde z_i \in \tilde U_- \subset \tilde V_i^+  &                      &  \rTo^{g^{-1}\circ\varphi_i^+}  &                 &  -z_i \in U_- \subset L_0    \\
\end{diagram}
\caption{Commutative diagram illustrating the proof of Proposition \ref{Proposition: meromorphic psi}}\label{Figure: Meromorphic psi}
\end{figure}

\begin{TP}\label{Proposition: meromorphic psi}
There is a pair of meromorphic functions on $L$ extending $\psi:L'\rightarrow G$, which has a pole of order $\pm v_i$ at $\pm z_i$. Furthermore, $O_i^-$ intersects $L$ at $-z_i$.
\end{TP}
\begin{proof}
To prove this lemma we use the action of $G$ to move a neighbourhood of $z_i \in L$ to a neighbourhood of the $G$-fixed point $z_i^{+}$. Then we are able to examine the behaviour of $\psi$ in the local coordinates found in Subsection \ref{Subsection: Local Model}. First we choose $g, h\in G$ so that $g.z_i$ and $h.(-z_i)\in V_i^+$. Then, we can choose a disc $U_+ \subset L$ about $z_i$ that contains no other special points and satisfies $g.U_+ \subset V_i^+$. We will denote $\tau(U_+)$ by $U_-$ and, shrinking $U_+$ is necessary, we will assume that $h.U_- \subset V_i^+$. If we denote $U_\pm\cap L'$ by $U_\pm'$, then $\psi(z).z\in U_-'$ for each $z\in U_+'$.

In $\tilde V_i^+$ the preimage of $g.U_+$ consist of $|\Gamma_i|$ disjoint holomorphic discs, each of which is mapped biholomorphically onto $g.U_+$ by $\varphi_i^+$. Let $\tilde U_+\subset \tilde V_i^+$ be one such disc and define $\tilde U_-$ similarly. Also, let $\pm \tilde z_i \in \tilde U_\pm$ be the unique preimage of $g.(\pm z_i)$ in $\tilde U_\pm$. A map $\tilde \psi : \tilde U_+' \rightarrow G_i$ is uniquely defined, in analogy to $\psi$, by
$$\tilde \psi(z).z \in \tilde U_-'.$$ Since the respective generators $s^{(1,0)}$ and $t^{(0,1)}$ for $G_i$ are mapped to $s^{u_{i-1}}$ and $t^{u_i}$ for $G$, $\tilde \psi = (\tilde \psi_1, \tilde \psi_2)$ satisfies \begin{equation}\label{Equation: tilde psi}(\tilde\psi_1(z))^{u_{i-1}}.(\tilde\psi_2(z))^{u_{i}}  = h\circ\psi\circ g^{-1}\circ\varphi_i^+(z) = h.g^{-1}\psi\circ\varphi_i^+(z) ,\end{equation}
for $z \in \tilde U_+$.

The intersection of $U_+$ with $O_i^+$ is transverse and therefore, $g.U_+$ intersects transversally with the special leaf $W'_{i}$. Consequently, a coordinate $v$ can be chosen on the disc $\tilde U_+$ with respect to which
$$\tilde U_+ = \{(\alpha(v), v, \beta(v)) \in \tilde V_i^+: |v|<\epsilon\},$$ where $\alpha, \beta$ are non-zero holomorphic functions and $v=0$ corresponds to $\tilde{z_i}$.
We also know that $U_-$ intersects with a special orbit transversally. If this special orbit is $O_i^-$, then $h.U_-$ intersect transversally with $W'_{i-1}$; otherwise, $h.U_-$ does not intersect with a special leaf in $V_i^+$. Thus, we can write
$$\tilde U_- = \{(a(u), b(u), c(u)) \in \tilde V_i^+: |u|<\epsilon\},$$ where $a, b, c$ are holomorphic functions and $u=0$ corresponds to $\tilde{z_i}$. The holomorphic functions $b$ and $c$ are non-zero, and $a$ has a zero of order one at $u=0$ if and only if $-z_i \in O_i^-$.

For each $v \in \tilde U_+'$ there is a unique $u \in \tilde U_-'$ such that $\tilde \psi = (\tilde \psi_1, \tilde \psi_2)$ satisfies $$(\tilde \psi_1(v).\alpha(v),\; \tilde \psi_2(v).v,\; \tilde \psi_1(v)\tilde \psi_2(v).\beta(v)) = (a(u), b(u), c(u)),$$ where the action of $G_i$ is given in (\ref{Equation: G-action}). It follows that \begin{equation}\label{Equation: psi lift}\tilde \psi(v) = (f(v)v,g(v)/v),\end{equation} where $f$ and $g$ are holomorphic and non-zero. Therefore, we can extend $\tilde\psi$ to a meromorphic function on $\tilde U_+$. If we take the quotient by $\varphi_i^+$, the pre-images of $g.U_+$ and $h.U_-$ are mapped biholomorphically onto $g.U_+$ and $h.U_-$, respectively. This gives a coordinate $v$ on $g.U_+$, with respect to which $v=0$ corresponds to $g.z_i$. Therefore, by equation (\ref{Equation: tilde psi}), $\psi$ has a pole of order $v_i$ at $0$. Note that we have also shown that $a$ vanishes at $v=0$; thus, the special orbit intersecting $L$ at $-z_i$ is $O_i^-$.

Using the same model, it follows that $\psi$ has a pole of order $-v_i$ at $-z_i$, which is in agreement with $\psi(-z) = \psi(z)^{-1}$. Since there are only a finite number of special orbits, $\psi$ has a finite number of poles and therefore, it is meromorphic.
\end{proof}

When $L \in \mo$, the real structure also commutes with $\psi$: $$(\psi\circ\gamma(z)).\gamma(z) = -\gamma(z) = \gamma(-z) = \gamma(\psi(z).z).$$ Therefore, $\psi$ additionally satisfies the reality condition \begin{equation}\label{Equation: Reality condition} \psi\circ\gamma = 1/\overline{\psi}.\end{equation}

\subsection{Constructing a Riemann surface in $Z$}\label{Subsection: Riemann surface}
In this subsection we define the ``square root" of $\psi$ and use it to construct a Riemann surface in $Z$. Most effort goes into establishing that this Riemann surface extends smoothly over points corresponding to special points. In particular the points on this surface corresponding to $z_i$ intersects with $L_i$ \emph{transversally}. This observation is important for constructing the microtwistor correspondence in Section \ref{Section: The Riemann surface R}.

We will continue to use the notation from the previous subsections. Therefore, $L \in \mCo$ is equipped with an involution $\tau$ and a meromorphic function $\psi$.
On a disc in $L'$ there is a well-defined square root of each of the components of $\psi = (\psi_1, \psi_2)$. So there is a well-defined square root of $\psi$, which we will denote by $\chi := (\psi_1^{1/2}, \psi_2^{1/2})$. We denote the abstract Riemann surface associated with the meromorphic continuation of $\chi$ by $R_L$. This is a branched covering $\rho: R_L \rightarrow L$ and the four possible values of $\chi$ over each point in $L$ determines to the deck transformation group \begin{equation}\label{Equation: H definition}H := \{(1,1), (-1,1), (1,-1), (-1,-1)\} \in G.\end{equation} In particular, $\rho$ has a double branch point over $z \in L$ if and only if $z = z_i$ or $- z_i$ and $$h_i := (-1)^{v_i} \neq (1,1).$$ The subset $R_L' : = \rho^{-1}(L') \subset R_L$ can be identified with the Riemann surface $$\{(z,w) \in L' \times G: w^2 = \psi(z)\}$$ and so, we can define a holomorphic map $$\mu: R_L' \rightarrow Z; (z, w) \mapsto w.z.$$

The next lemma can be thought of as a ``multi-valued version" of Proposition \ref{Proposition: meromorphic psi} because it extends $\mu$ to a holomorphic map over $R_L$. Consequently, the method of proof is similar and we will use some of the same notation.

\begin{TL}\label{Lemma: Riemann surface from R_L}
$\mu:R_L' \rightarrow Z'$ extends to a holomorphic map on $R_L$. In particular, $\mu$ is biholomorphic on some neighbourhood of $r \in  \rho^{-1}(\pm z_i)$ and $\mu(r) \in L_i$.
\end{TL}
\begin{proof}
Let $U$ be a connected component of $\rho^{-1}(U_+)$ in $R_L$. The map $$\rho|_{U}: U \rightarrow U_+,$$ which is labeled by $(i)$ in Figure \ref{Figure: Riemann surface proof}, is a biholomorphism if $h_i = (1,1)$, otherwise it is a double cover branched over $z_i$. Thus, $U$ is an open disc about $\rho^{-1}(z_i)\cap U$ in $R_L$. As in the proof of Proposition \ref{Proposition: meromorphic psi}, we choose $g\in G$ such that $g.U_+ \subset V_i^+$, where $\varphi_i^+: \tilde V_i^+ \rightarrow V_i^+$ is an admissible neighbourhood of $z_i^+$. Also, $\tilde U_+$ denotes a lift of $g.U_+$ to $\tilde V_i^+$ and $\tilde z_i$ denotes the lift of $g.z_i$ to $\tilde U_+$.

Let $\tilde \rho :\tilde U\rightarrow \tilde U_+$ be the double cover of $\tilde U_+$ branched over $\tilde z_i$. The induced map between $\tilde U$ and $U$, which is labeled $(ii)$ in Figure \ref{Figure: Riemann surface proof}, is a double cover branched over $\rho^{-1}(z_i)\cap U$ if $h_i = (1,1)$ and otherwise, it is biholomorphic.

\begin{figure}[htbp]
\begin{diagram}
 \tilde U_+ \subset \tilde V_i^+ & \rTo^{\varphi_i^+ \quad\quad}  &  g.U_+ \subset V_i^+  &  \rTo^{\quad\quad g^{-1}}  &  U_+ \subset L      \\
\uTo^{\tilde\rho}                        &                      &                                 &                 &  \uTo_{\rho}^{(i)}  \\
\tilde U                              &          &             \rDashto^{(ii)}                  &                 &  U  \subset R_L  \\
\dTo^{\tilde \mu: u \mapsto \tilde\chi(u).\tilde\rho(u)}  &&& & \dTo_{\mu}   \\
\tilde \mu(\tilde U) \subset \tilde V_i^+  &&\rTo^{g^{-1}\circ\varphi_i^+}_{(iii)}  &&  \mu(U) \subset Z  \\
\end{diagram}
\caption{Commutative diagram illustrating the proof of Lemma \ref{Lemma: Riemann surface from R_L}}\label{Figure: Riemann surface proof}
\end{figure}

Recall from Lemma \ref{Proposition: meromorphic psi} that $\tilde U_+$ can be parameterized in the coordinates on $V_i^+$ as
$$\{(\alpha(v), v, \beta(v)) \in \tilde V_i^+: |v|<\epsilon\},$$ where $\alpha, \beta$ are non-zero holomorphic functions and $v=0$ corresponds to $\tilde z_i$.
We can define a coordinate $u$ on $\tilde U$ by setting $$u^2 = v,$$ and with this coordinate $u=0$ corresponds to $\tilde\rho^{-1}(\tilde z_i)$.  Then, by taking the square root of equation (\ref{Equation: psi lift}), $\tilde \chi := {\tilde \psi}^{1/2}$ can be written as $$\tilde\chi: \tilde U' \rightarrow G_i; \;\;u \mapsto (f(u).u, g(u)/u),$$ for some holomorphic and non-vanishing functions $f$ and $g$.

With the action of $G_i$ described in (\ref{Equation: G-action}), the map $\tilde \mu$ is given by
$$\tilde \mu: \tilde U \rightarrow \tilde V_i^+;$$ $$u \mapsto \tilde \chi(u).(\alpha(u^2), u^2, \beta(u^2)) = (a(u).u, b(u).u, c(u)),$$
where $a,b$ and $c$ are non-zero holomorphic functions. Thus, $\tilde \mu$ is a holomorphic map on $\tilde U$. Since $a$ and $b$ are non-vanishing at $u=0$, $\tilde \mu$ is biholomorphic in a neighbourhood of $0$. Therefore, by shrinking $U_+$ if necessary, we can assume that $\tilde \mu$ is a biholomorphism on $\tilde U$.

By the definition of $G_i$, $(-1,-1) \in G_i$ is mapped to $(-1)^{u_i + u_{i-1}} \in G$. Therefore, $(-1,-1) \in G_i$ is mapped to $(1,1) \in G$ if and only if $h_i = (1,1)$. In this case the restriction of $\varphi_i^+$ to $\tilde \mu(\tilde U) \subset \tilde V_i^+$ is a double cover of its image branched over $(0,0,c(0))$, since $$\varphi_i^+(a(u).u, b(u).u, c(u)) = \varphi_i^+(-a(u).u, -b(u).u, c(u));$$ otherwise, $\varphi_i^+|_{\tilde \mu (\tilde U)}$ is a biholomorphism.

Therefore, when $h_i \neq (1,1)$, arrows $(ii)$ and $(iii)$ correspond to biholomorphisms and it follows that $u \mapsto \chi(u).\rho(u)$ is a biholomorphism. When $h_i = (1,1)$, $g^{-1}\varphi_i^+ \circ \tilde\mu$  gives a holomorphic double cover of $\mu(U)$ branched at $\mu(0)$ and the induced map from $\tilde U \rightarrow U$ (labeled (ii)) is a double cover branched over $0 \in U$; thus, $\mu$ extends to a biholomorphism on $U$. In either case $\mu$ is not branched over $\mu(0)$, and $\mu(0) \in L_i$. For another connected component of $\rho^{-1}(U_+)$ in $R_L$, we can apply the same argument to extend $\mu$ to $\rho^{-1}(U_+)$. Then, we can repeat this procedure for each $\pm z_i$ to complete the proof.
\end{proof}

Thus, $R_L$ is mapped holomorphically by $\mu$ onto a curve in $Z$. This curve will be examined in more detail in Section \ref{Section: The Riemann surface R}. The results in this section enable us to conclude with following lemma about special points.

\begin{TL}\label{Lemma: Special points}
The special orbits $O_i^\pm$ intersect $L \in \mo$ at distinct points $\pm z_i$ with $|z_i| = 1$.
\end{TL}
\begin{proof}
By Lemma \ref{Lemma: Special orbit intersection} we have $O_i^+ = \gamma(O_i^-)$. Then, by Proposition \ref{Proposition: meromorphic psi}, we have $$-z_i = \gamma(z_i) = -\frac{1}{\bar{z}_i}$$ and so, $|z_i|  = 1$ for $i = 1, \ldots, k$.

By Lemma \ref{Lemma: Riemann surface from R_L}, $\mu$ maps a point in $\rho^{-1}(\pm z_i)$ to a point in $L_i$; thus, $z_1, \ldots, z_k$ are distinct, since $L_1, \ldots, L_k$ are distinct.
\end{proof}

\section{The conformal structure}\label{Section: The conformal structure}
\subsection{Coordinates on principal lines}\label{Subsection: Coordinates}
In this subsection we single out a principal line $L_0$ and then parameterize principal lines in $Z$ using two ($\tau$ invariant) pairs of points on this line. The parameterization is depicted in Figure \ref{Figure: Parameterizing the principal lines}. Then in Subsection \ref{Subsection: principal lines as graphs} we are able to write down a meromorphic function representing the action of $G$ required to travel from $L_0$ to $L$ as shown in Figure \ref{Figure: Principal lines as graphs}. Finally in Subsection \ref{Subsection: Parameterizing principal lines} we are able to recover to conformal structure on these twistor lines by explicitly calculating null sections in the normal bundles of principal lines.

As noted in Remark \ref{Remark: Coordinates} the coordinate on $L \in \mCo$ is not unique. There are two permissible types of coordinate transformation: one is $z \mapsto \lambda.z$ for $\lambda \in \C^*$, and the other is \begin{equation}\label{Equation: swap tangency map}z \mapsto z^{-1}.\end{equation} We can eliminate the rotational freedom about $0$ and $\infty$ by fixing $z_1 = 1$. Then, since $z_2 \neq 1,-1, 0$ or $\infty$, we can assume that $z_2$ is contained in \begin{equation}\label{Equation: unique coordinate set}\{z \in \C : \tn{arg}(z) \in (0, \pi)\}\cup (-1,0) \cup (1, \infty),\end{equation} by performing the coordinate change (\ref{Equation: swap tangency map}) if necessary. When $L \in \mo$, $|z_i| = 1$, by Lemma \ref{Lemma: Special points}. Therefore, a coordinate change setting $z_1 = 1$ preserves the real structure on $L$, as noted in Remark \ref{Remark: Coordinates}. These observations are summarized in the next lemma.

\begin{figure}[htbp]
  \psfrag{A}{$0$}
  \psfrag{B}{$\infty$}
  \psfrag{C}{$a^{1/2}$}
  \psfrag{D}{$-a^{1/2}$}
  \psfrag{E}{$b^{1/2}$}
  \psfrag{F}{$-b^{1/2}$}
  \psfrag{G}{$L$}
  \psfrag{H}{$L_0$}
  \centering \scalebox{1.0}{\includegraphics{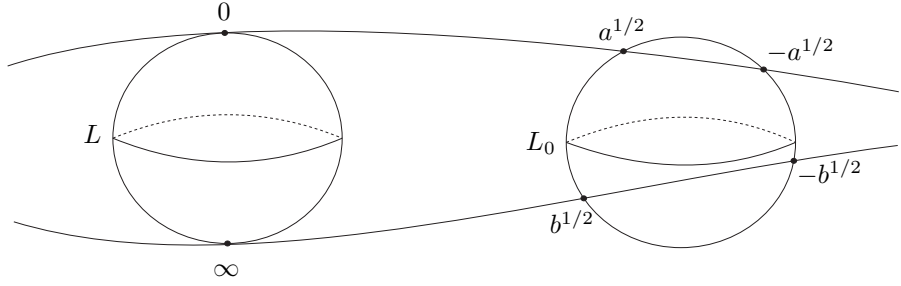}}
  \caption{Parameterizing the principal lines with points on $L_0$}\label{Figure: Parameterizing the principal lines}
\end{figure}

\begin{TL}\label{Lemma: Unique coordinate}
There is a unique affine coordinate $z$ on $L \in \mCo$ such that: the holomorphic involution defined in Subsection \ref{Subsection: involution} can be written as $\tau: z \mapsto -z$, $z_1 = 1$, and $z_2$ belongs to (\ref{Equation: unique coordinate set}). Furthermore, when $L \in \mo$ the real structure is given by $z \mapsto -\bar{z}^{-1}$.
\end{TL}

In the remainder of this article we will single out $L_0 \in \mo$ and we will denote the coordinate on $L_0$ satisfying Lemma \ref{Lemma: Unique coordinate} by $z$. On any other $L \in \mCo$ we will denote the coordinate satisfying Lemma \ref{Lemma: Unique coordinate} by $w$. We will respectively denote the involutions on $L_0$ and $L$ by $\tau_0$ and $\tau$; the meromorphic functions on $L_0$ and $L$, which were defined in Subsection \ref{Subsection: Meromorphic data}, by $\psi_0$ and $\psi$; and the intersection point of $O_i^+$ with $L_0$ and $L$ by $z_i$ and $w_i$. There is a pair of (non-special) $G$-orbits that intersect $L$ tangentially at $0$ and $\infty$, and we will denote their intersection points with $L_0$ by $\pm z_0$ and $\pm z_\infty$, respectively.

We can write the quotient map from $L_0$ to $L_0/\tau_0$ as $z \rightarrow z^2$ and similarly for $L \rightarrow L/\tau$. There is a well-defined holomorphic map $$F: L_0/\tau_0 \rightarrow L/\tau,$$ sending $z^2$ to $w^2$, where $w$ is in the same $G$-orbit as either $z$ or $-z$ (or both). This map is biholomorphic with $F(z_0^2) = 0$, $F(z_\infty^2) = \infty$ and therefore, $$F(z^2) = c.\frac{z^2-a}{z^2 - b},$$ where $a = z_0^2$, $b = z_\infty^2$ and $c\in \C^*$. Since $F(1) = 1$, we require $$c := \frac{1 - b}{1 - a}.$$ By interchanging $a$ and $b$ if necessary, we can ensure that $w_2$ belongs to (\ref{Equation: unique coordinate set}).

\subsection{The principal lines as graphs}\label{Subsection: principal lines as graphs}
Using $F: L_0/\tau \rightarrow L/\tau$, which was defined in the previous subsection, we can define a complex curve by $$E := \{(z,w) : w^2 = F(z^2)\} \subset L_0 \times L.$$ Recalling that $L_0'$ denotes the subset of $L_0$ where $\psi_0$ is holomorphic (or equivalently the points of $L_0$ on non-special orbits), we define $E' = \{(z,w) \in E : z \in L_0'\}.$ We can define a $G$-valued holomorphic function $f$ on $E'$ satisfying \begin{equation}\label{Equation: f definition}f(z,w).z = w.\end{equation} It follows from this definition that \begin{equation}\label{equation: f under tau}f(z,w) = f(-z, w).\psi_0(z)\end{equation} and \begin{equation}\label{equation: f under tilde tau}f(z,w) = f(z,-w).\psi(-w).\end{equation}

\begin{figure}[htbp]
  \psfrag{L0}{$L_0$}
  \psfrag{L}{$L$}
  \psfrag{A}{$\psi(z)$}
  \psfrag{B}{$\psi(w)$}
  \psfrag{C}{$f(z,-w)$}
  \psfrag{D}{$f(z,w)$}
  \psfrag{E}{$f(-z,w)$}
  \psfrag{F}{$f(-z,-w)$}
  \psfrag{G}{$z$}
  \psfrag{H}{$-z$}
  \psfrag{I}{$w$}
  \psfrag{J}{$-w$}
  \centering \scalebox{1.0}{\includegraphics{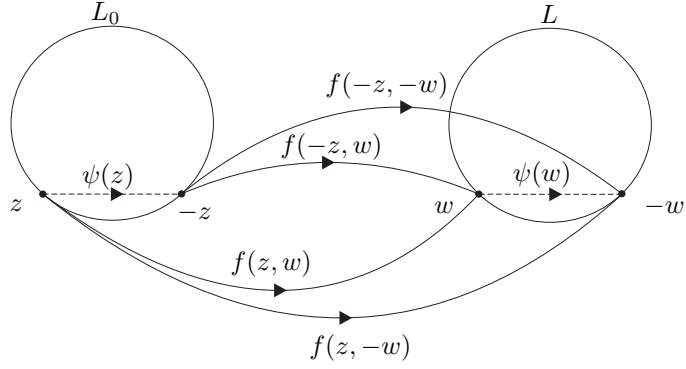}}
  \caption{Principal lines as graphs}\label{Figure: Principal lines as graphs}
\end{figure}


For the next lemma we use the notation $\chi := \psi^{1/2}$ and $\chi_0 := \psi_0^{1/2}$ from the previous section.
\begin{TL}\label{Lemma: Principal line graph}
If $f$ is defined by equation (\ref{Equation: f definition}), then $f$ extends to a meromorphic function on $E$ satisfying \begin{equation}\label{Equation: f in terms of psi}f(z,w) = c.\chi_0(z).\chi(-w),\end{equation} for some constant $c \in G$.
\end{TL}
\begin{proof}
We first show that $f$ extends meromorphically to $E$. If $z_0\in L_0$ is in a special orbit, then $f$ is holomorphic at either $(z_0, w_0)$ and $(-z_0, -w_0)$ or $(-z_0, w_0)$ and $(z_0, -w_0)$ in $E$. Suppose $f$ is holomorphic at $(z_0, w_0) \in E$ (the argument proceeds similarly if we make the other choice). Then, by equation (\ref{equation: f under tau}), $f(z,w).\psi_0(-z)$ is holomorphic at $(-z_0, w_0)$ and so, $f$ has a meromorphic singularity at $(-z_0, w_0)$. Since $f$ is not holomorphic at $(-z_0, w_0)$, it follows that $f$ must be holomorphic at $(-z_0, - w_0)$. Then, by equation (\ref{equation: f under tau}) again, $f(z, w).\psi_0(-z)$ is holomorphic at $(z_0, -w_0)$ and so, $f$ has a meromorphic singularity at $(z_0, -w_0)$. Thus, $f$ is meromorphic on $E$ and in particular, the orders of the singularities of $f$ are such that \begin{equation}\label{Equation: ffff holomorphic}f(z,w).f(-z,w).f(z,-w).f(-z,-w),\end{equation} is $G$-valued.

From equations (\ref{equation: f under tau}) and (\ref{equation: f under tilde tau}) we obtain \begin{equation}\label{Equation: f first eq}f(z,w)^2 = f(z,w).f(-z,-w).\psi_0(z).\psi(-w).\end{equation} If we define $A(z,w) := f(z,w)f(-z,-w)$, then $A(z,w) = A(z,-w)$, since $$A(z,w) = f(-z,-w).f(z,w)\psi_0(z).\psi_0(-z) = f(z,-w).f(-z,w).$$ Therefore, $A$ defines a meromorphic function on $L_0$. This function on $L_0$ is $G$-valued, since $$A(z,w)^2 = A(z,w).A(z,-w)$$ is equivalent to the expression in equation (\ref{Equation: ffff holomorphic}). However, any $G$-valued holomorphic function on $L_0$ is constant; thus, by equation (\ref{Equation: f first eq}), $$f(z,w)^2 = c^2.\psi_0(z).\psi(-w),$$ for some non-zero constant.
\end{proof}

\subsection{Parameterizing the principal lines}\label{Subsection: Parameterizing principal lines}

If there exists $L \in \mo$ corresponding to $a, b \in L_0/\tau_0$ in the way described, then it is determined relative to $L_0$ by equation (\ref{Equation: f in terms of psi}), for some $c \in G$. Distinct choices of $c$ correspond to distinct twistor lines. We will denote the principal line corresponding to $a,b$ and $c$ by $c.L_{ab}$. Note that $c.L_{ab}$ and $c.L_{ba}$ are the same twistor line. Thus, we have well-defined injective map $$\Phi:\mCo \rightarrow \Delta_\C \times G;$$$$c.L_{ab} \mapsto (a,b,c),$$ where $$\Delta_\C := \{(a,b) \in \Sigma^2 (L_0/\tau_0): a \neq b\}$$ and $\Sigma^2$ denotes the symmetric product. We will denote the inverse of $\Phi$ by $$\Psi: \Phi(\mCo) \subset  \Delta_\C \times G \rightarrow \mCo.$$ The map $\Psi$ sends $(a,b,c)\in\Phi(\mCo)$ to the twistor line
$$c.L_{ab} = \{f(z,w).z : (z,w) \in E\}.$$ This map can be written down explicitly, since $f(z,w) = c.\chi(-w).\chi_0(z)$ by Lemma \ref{Lemma: Principal line graph}, and Proposition \ref{Proposition: meromorphic psi} implies \begin{equation}\label{Equation: psi formula}\chi^2(w) := \psi(w) = \prod_{i=1}^{k}\Big(\frac{w+w_i}{w-w_i}\Big)^{v_i},\end{equation}
while Subsection \ref{Subsection: Coordinates} implies $$w^2 = \frac{1 - b}{1 - a}.\frac{z^2-a}{z^2 - b} \quad \tn{ and }\quad w_i^2 = \frac{1 - b}{1 - a}.\frac{z_i^2-a}{z_i^2 - b}.$$ Consequently $\Psi$ is smooth. In the remainder of this subsection we differentiate $\Psi$ in order to show that $\Phi$ is a diffeomorphism.

Let $(a_0,b_0,c_0) \in \Phi(\mCo)$ correspond to the principal line $L$ with coordinate $w$. On $c.L_{ab}$ we will denote the coordinate by $\tilde{w}$, the meromorphic function by $\psi_{ab}$ and the square root of $\psi_{ab}$ by $\chi_{ab}$. By Lemma \ref{Lemma: Principal line graph}, the ``graph" of $c.L_{ab}$ in $L\times G$ is given by
$$\big(w,\, c.\chi_{ab}(-\tilde{w}).\chi(w)\big).$$
By differentiating this expression in terms of $(a,b,c)$ we find that the derivative of $\Psi$ at $L$
$$\tn{d}\Psi: T_{(a_0,b_0,c_0)}\Phi(\mCo) \rightarrow T_{L}\mCo \cong H^{0}(N_{L})$$
maps $$a'\frac{\partial}{\partial a} + b'\frac{\partial}{\partial b} + c_1'\frac{\partial}{\partial c_1} + c_2'\frac{\partial}{\partial c_2}$$
to the section of $H^{0}(N_{L})$ given by
\begin{equation}\label{Equation: Phi derivative}\Big(c.\big(\partial_a(\phi_{ab}(\tilde{w})).a' + \partial_b(\phi_{ab}(\tilde{w})).b'\big) + c'\Big)_{(a_0,b_0,c_0)}\partial_{g},\end{equation}
where $\phi_{ab} := \tn{log}\chi_{ab}$. Substituting the formula for $\psi$ given in equation (\ref{Equation: psi formula}) into equation (\ref{Equation: Phi derivative}) gives
\begin{equation}\label{Equation: Phi derivative two}\big(2\pi c.(A.a'.\tilde w^{-1} + B.b'.\tilde w) + c'\big)_{(a_0,b_0,c_0)}\partial_{g},\end{equation} where $$A(a,b) := \frac{1}{4\pi (b-a)}\sum_{i=1}^k v_i.w_i^{-1},\quad B(a,b) := \frac{1}{4\pi (b-a)}\sum_{i=1}^k v_i.w_i.$$

\begin{TL}
There is a smooth parameterization of the principal lines given by
$$\Phi:\mCo \rightarrow \Phi(\mCo) \subset \Delta_\C \times G; \; c.L_{ab} \mapsto (a,b,c).$$
\end{TL}
\begin{proof}
It follows from (\ref{Equation: Phi derivative two}) that the kernel of $\tn{d}\!\Psi$ at $(a_0,b_0,c_0)$ is trivial provided $A$ and $B$ are both non-zero. One can easily check that $$A = \lambda. \psi'(0),$$ for some $\lambda \in \C^*$. Then note that the action of $G$ is tangential to $L$ at $0$ and so, the tangent plane to $L$ at $0$ is spanned by $\psi'(0) \partial_g$. Therefore, $A \neq 0$ and similarly, $B \neq 0$. We noted in Subsection \ref{Subsection: Twistor spaces} that $\tn{dim}_{\C}H^{0}(N_{L}) = 4$ and so, $\tn{d}\!\Psi$ is an isomorphism. Thus, the inverse function theorem completes the proof.
\end{proof}

\subsection{The conformal structure}\label{subsection: Conformal structure}
Recall that the infinitesimal deformations of $L \in \mCo$ that intersect with $L$ define the \emph{null cone} of the conformal structure on $T_L \mCo$. These are precisely the sections of $N_{L}$ that vanish at a single point; by setting both components of (\ref{Equation: Phi derivative two}) to zero, we can determine sections in the null cone.

\begin{TL}\label{Lemma: non-degeneracy}
If $L \in \mCo$, then $(A_1B_2 - B_1A_2) \neq 0.$
\end{TL}
\begin{proof}
Since $A$ and $B$ are non-zero, $(A_1B_2 - B_1A_2) = 0$ implies $A = \kappa.B$, for some $\kappa \in \C^*$. Therefore, the section given in (\ref{Equation: Phi derivative two}) can be written as
\begin{equation}\label{Equation: degenerate normal bundle}\big(2\pi c.A.(a'.w^{-1} + \kappa.b'.w) + c'\big)\partial_{g}.\end{equation}
If $c^{-1}.c' = \lambda.A$ for some $\lambda \in \C$, then (\ref{Equation: degenerate normal bundle}) has two zeros on $L$. It follows that $N_L \ncong \mc{O}(1) \oplus \mc{O}(1)$ and therefore, $L \notin \mCo$.
\end{proof}

Vectors in the null cone at $L$ satisfy the simultaneous equations
$$2\pi c_1(A_1a' w^{-1} + B_1 b' w) + c_1' = 0$$ $$2\pi c_2(A_2 a' w^{-1} + B_2 b' w) + c_2'  = 0.$$
$$$$
Since $(A_1B_2 - B_1A_2) \neq 0,$ $w$ can be eliminated to show that the conformal structure on $T_L \mCo$ is determined by the representative
\begin{equation}\label{Equation: conformal structure}\tn{d}a\tn{d}b - \frac{(A_2c_1^{-1}\tn{d}c_1 - A_1c_2^{-1}\tn{d}c_2)(B_1c_2^{-1}\tn{d}c_2 - B_2c_1^{-1}\tn{d}c_1)}{(2\pi)^2(A_1B_2 - B_1A_2)^2}.\end{equation}

\begin{TL}\label{Lemma: Real principal lines}
The real principal lines correspond to the submanifold
$$\Phi(\mo) = \Delta^\circ \times F,$$ where $\Delta^\circ : = \{(a,b) \in \Delta_\C: b = \bar{a}^{-1}\}$.
\end{TL}
\begin{proof}
Let $L \in \mo$ correspond to $(a,b,c) \in \Delta_\C \times G$. The $G$-orbits that intersect with $L$ tangentially are interchanged by the real structure. Therefore, $b = \bar{a}^{-1}$.

The two points on $L$ where the action of $G$ is tangential are given by
$$f(a,0).z = c.\chi_0(a).a,$$
and
$$f(\gamma(a),0).z = c.\chi_0(\gamma(a)).\gamma(a),$$
for the appropriate of branch of $\chi_0$.
Since $\gamma$ satisfies equation (\ref{Equation: gamma anticommutes}):
$$c.\chi_0(a).a  = \gamma(c.\chi_0(\gamma(a)).\gamma(a)) = \big(\overline{c.\chi_0(\gamma(a))}\big)^{-1}.a.$$
By applying the reality condition on $\psi$, which was given in (\ref{Equation: Reality condition}), it follows that $c \in F$.

Now we need to show that every $(a, \bar{a}^{-1}, c) \in \Delta^\circ\times F$ corresponds to some $L \in \mo$. Using the function $f$ determined in Lemma \ref{Lemma: Principal line graph} we obtain a rational curve $L$ embedded in $Z$ that is preserved by $\gamma$ and therefore, $L\in \m$. By construction the action of $G$ is free about $L$ and so, $L \in \mo$.
\end{proof}

If we restrict equation (\ref{Equation: conformal structure}) to $\Delta^\circ\times F$ we obtain the conformal structure on $M^\circ \cong \mo$. We will use the coordinate $(\theta_1, \theta_2)$ on $F$. These coordinates can be expressed in terms of $c \in F$ as $c_j = e^{2\pi i\theta_j}$, for $j = 1,2$. The functions $A$ and $B$ are complex valued on $\Delta^\circ$ and satisfy $$\overline{(b-a).B} = (b-a).A.$$ We will express $A$ and $B$ in terms of real functions by defining setting
$$Q + iP = (b-a).A.$$ Then, by restricting (\ref{Equation: conformal structure}) to $\Delta^\circ\times F$, we obtain the next lemma.

\begin{TL}\label{Lemma: conformal structure}
The conformal structure on $M^\circ \cong \Delta^\circ \times F$ is determined by the representative
\begin{equation}\label{Equation: real conformal structure}\frac{\tn{d}a\tn{d}\bar{a}}{(|a|^2-1)^2} + \frac{(P_2\tn{d}\theta_1 - P_1\tn{d}\theta_2)^2 + (Q_2\tn{d}\theta_1 - Q_1\tn{d}\theta_2)^2}{4(P_1Q_2 - P_2Q_1)^2},\end{equation}
where $P$ and $Q$ are real valued and satisfy $$(P_1Q_2 - P_1Q_2) \neq 0.$$
\end{TL}
\begin{TR}\label{Remark: psi determines conformal}When $L = \Psi(a,b,c) \in \mo$ we can assume $|a| < 1$ and so, $|b| = |\bar{a}^{-1}|> 1$. In this case, the unique coordinate determined in Lemma \ref{Lemma: Unique coordinate} can be written in terms of $z$ as $$w^2 = \frac{1 - b}{1 - a}.\frac{z^2-a}{z^2 - b}.$$ So, $w_1, \ldots, w_k$ are uniquely determined by $z_1, \ldots, z_k$. Thus, the conformal structure on $M^\circ$ is determined by the combinatorial data $\mc{T}$ and the special points $z_1, \ldots, z_k$ on $L_0$.\end{TR}

\section{The microtwistor correspondence}\label{Section: The Riemann surface R}
\subsection{A Riemann surface}\label{Subsection: Twistor lines relative to R}
In this section we give a correspondence between the microtwistor space, which is a subset of $\tau$ invariant pairs of points on $L_0$ and $N$, the $F$-orbit space of $M$; this is what we refer to as the microtwistor correspondence. The map between the microtwistor space and $N$ is composed of a map to the Riemann surface in $Z$, which was constructed in subsection \ref{Subsection: Riemann surface}, followed by a map to $N$. These maps are delineated in Figure \ref{Figure: Microtwistor}.

\begin{figure}[htbp]
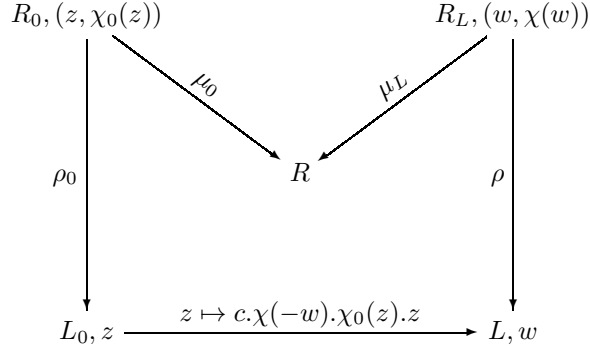

\begin{diagram}
R_0, (z, \chi_0(z)) &&&& R_L, (w, \chi(w))\\
&\rdTo^{\mu_0}&&\ldTo^{\mu_L}&\\
\dTo^{\rho_0}&&R&& \dTo^{\rho}\\
&&&&\\
L_0, z&&\rTo^{z \mapsto c.\chi(-w).\chi_0(z).z}&& L, w\\
\end{diagram}
\caption{Principal lines in relation to the curve $R$}\label{Figure: Twistor lines relative to R.}
\end{figure}

\begin{TD}Let $L \in \mCo$ correspond to $(a,b,c) \in \Delta_\C \times G$. In Subsection \ref{Subsection: Riemann surface} we defined a Riemann surface $R_L$ associated with the multivalued function $\chi$ on $L$. Away from the special points, a point on $R_L$ can be denoted by $(w,\chi(w))$. The deck transformation group of the covering $\rho: R_L \rightarrow L$ is given by $H\subset F$, where $H$ was defined in (\ref{Equation: H definition}). In Lemma \ref{Lemma: Riemann surface from R_L} we showed that the map $$\mu_L: (w,\chi(w)) \mapsto c^{-1}.\chi(w).w$$ extends to a holomorphic map over $R_L$, and we will denote the image of $\mu_L$ by $R$. We will also define a closed disc in $R$ by $$D := \mu_L\big(\{(w,\chi(w)) \in R_L: |w| \leq 1\}\big),$$ on the branch where $\chi(0) = (1,1)$. It follows from equation (\ref{Equation: f in terms of psi}) that these definitions do not depend on the choice of $L$.\end{TD}

For $L = \Psi(a,\bar{a}^{-1}, (1,1)) \in \mo$ (with $|a| < 1$), we will denote $0$ and $\infty \in L$ by $0_a$ and $\infty_a$, respectively. Hence, $0_a$ and $\infty_a$ are the two points in $R \cap L$ and so, $$R \cap (F.L) = H.\{0_a, \infty_a\}.$$ For the proof of the next lemma it is necessary to observe that these are the branch points of $\mu_L$. From this perspective, $R$ can be thought as a ``Riemann surface of tangency points" belonging to a family of such surfaces parameterized by $F/H$. This perspective is illustrated in Figure \ref{Figure: Riemann surface of tangency points}.

\begin{figure}[htbp]
  \psfrag{A}{$L_0$}
  \psfrag{B}{$L$}
  \psfrag{C}{$a^{1/2}$}
  \psfrag{D}{$-a^{1/2}$}
  \psfrag{E}{$0$}
  \psfrag{F}{$\chi_0(a^{1/2}).a^{1/2}$}
  \psfrag{G}{$R$}
  \psfrag{H}{$G.a^{1/2}$}
  \centering \scalebox{1.0}{\includegraphics{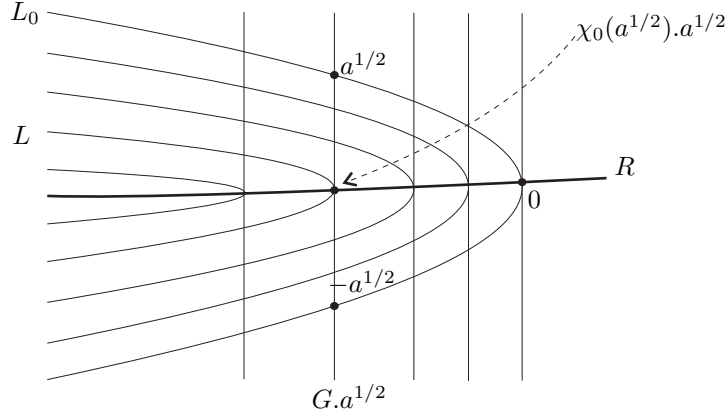}}
  \caption{$R$ as a Riemann surface of tangency points}\label{Figure: Riemann surface of tangency points}
\end{figure}

\begin{TL}\label{Lemma: nu definition}
The conformal isometry $$\Delta^\circ \rightarrow D^\circ\subset R;\; (a,\bar{a}^{-1}) \mapsto 0_a$$ extends to a conformal isometry between $\Delta$ and $D$, which we will denote by $\nu$.
\end{TL}
\begin{proof}
For $z \in L_0$, $(z^2, \bar{z}^{-2})$ is mapped to $\chi_0(z).z$ by $\nu$. So, by the definition of $D$, $\nu$ is well-defined and surjective. Also $\nu$ is injective, since $\nu(z^2, \bar{z}^{-2})$ is the only point in $D$ contained in $G.z$.

Suppose that $r \in R$ is not in $L_1, \ldots, L_k$, and choose $L = \Psi(a,b, c)\in \mo$ with $r \notin H.L$. By our choice of $L$, we can choose an open disc $U$ about $w_0 \in L \cap G.r$ that does not contain $0_a$ or $\infty_a$. Therefore, $$U\times G \rightarrow Z; (w,g) \mapsto g.w$$ is a biholomorphism onto its image and so, we can use the coordinates $(w,g) \in U \times G$ on a neighbourhood of $r$ in $Z$. Hence, we can use the coordinates $(w, \chi(w))$ on a neighbourhood of $r$ in $R$.

So, $w \rightarrow (w, \chi(w))$ is a biholomorphism on $U$ and, by our choice of $U$, $w^2 \rightarrow (w,\chi(w))$ is also a biholomorphism. If we compose this map with the biholomorphism $z^2 \rightarrow w^2$ (defined in Subsection \ref{Subsection: Coordinates}), then we obtain a local biholomorphism from $L_0/\tau_0$ to $R$, away from $z_1^2, \ldots, z_k^2$.

We have shown that, away from $(z_i^2, z_i^2)$, $\nu$ is locally the restriction of a biholomorphism to a closed disc. In Lemma \ref{Lemma: Riemann surface from R_L} it is shown that the same is true in a neighbourhood of $(z_i^2, z_i^2)$. This completes the proof, since $\nu$ is bijection.
\end{proof}

\begin{TL}\label{Lemma: Delta properties}
The boundary of $D$ is contained in $\pi^{-1}(B)$. In particular, $$r_i := \nu(z^2_i, z_i^2)\in L_i,$$ for $i = 1, \ldots, k$.
\end{TL}
\begin{proof}
We know from Lemma \ref{Lemma: Riemann surface from R_L} that $r_i \in L_i$, for $i = 1,\ldots, k$. So suppose that $z \in L_0$ is not contained in a special orbit and $|z| = 1$; thus, $z$ is contained in a $B$-orbit. We will assume that $r := \chi_0(z).z \notin \pi^{-1}(B)$ in order to obtain a contradiction.

By our assumption $r$ is contained in some $L := \Psi(a,b,c) \in \mo$. Therefore, $$r^2 =  \frac{1 - b}{1 - a}.\frac{z^2-a}{z^2 - b}$$ and consequently, $|r| = 1$. Then, by Lemma \ref{Lemma: Principal line graph}, $$r = c.\chi_0(z).\chi(-r).z,$$ for some $c \in F$. Since $r = \chi_0(z).z$ and $\psi = \chi^2$, $$\psi(r) \in F.$$ It follows that $\psi(r) = (1,1)$, as otherwise $-r = \psi(r).r \notin L$. However, $\psi(r) = (1,1)$ only when $r = 0$ or $\infty$. This gives a contradiction since, $|r| = 1$.
\end{proof}

We will denote the $F$-orbit space of $\m$ by $\n$, so there is a conformal isometry between $\n$ and $N$ induced from the conformal isometry between $\m$ and $M$.  There is an induced map between $D$ and $\n$, which we denote by $\sigma$, that maps a point in $D$ to the $F$-orbit of principal lines containing it.
\begin{TL}\label{Lemma: D intersection}
The map $\sigma: D \rightarrow \n$ is injective.
\end{TL}
\begin{proof}
There is a single point $0_a$ in an $F$-orbit of principal lines intersecting $D^\circ$. Also, Lemma \ref{Lemma: Riemann surface from R_L} shows that $D$ intersects $L_i = F.L_i$ once at $r_i$. So suppose that $r,r' \in \partial D$ are distinct from  $r_1, \ldots, r_k$ and are both contained in $n \in \partial\n$. From the definition of $D$, it is clear that the action of $G$ about $r$ and $r'$ is free and so, they are contained in $\pi^{-1}(B)\backslash \Sigma$. Hence, the closure of the $G$-orbit of $r$ contains $n$. Consequently, $r$ and $r'$ are contained in the same $G$-orbit. Since distinct points in $D$ cannot be contained in the same $G$-orbit, it follows that $r = r'$.
\end{proof}

\subsection{The microtwistor correspondence}\label{Subsection: Defining S}
The $G$-orbit through $0_a \in D$ intersects $L := \Psi(a,b,c) \in \mo$ tangentially. Also, the twistor line containing a point in $\partial D$ is contained in $\pi^{-1}(B)$ and therefore, the closure of some $G$-orbit. Thus, the $G$-orbit through each point in $D$ is tangential to the corresponding twistor line at that point. The same is true about any other point in $R$. Therefore, to prove that the intersection of $R$ with twistor lines is transverse, it suffices to show that $R$ intersects $G$-orbits transversally. As shown in Subsection \ref{Subsection: Twistor lifts}, a complex curve in $Z$ that intersects with twistor lines transversally is projected conformally by $\pi$ onto a surface in $M$. The next lemma proves that $R$ is such a curve.

\begin{figure}[htbp]
  \psfrag{A}{$\Delta$}
  \psfrag{B}{$L_0$}
  \psfrag{C}{$D \subset R$}
  \psfrag{D}{$N$}
  \psfrag{E}{$\mu_0$}
  \psfrag{F}{$\sigma$}
  \psfrag{G}{$\nu$}
  \centering \scalebox{1.0}{\includegraphics{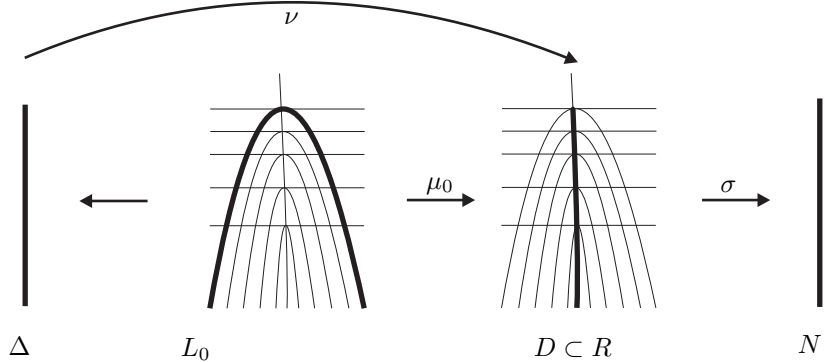}}
  \caption{The microtwistor correspondence}\label{Figure: Microtwistor}
\end{figure}

\begin{TL}\label{Lemma: D transverse}
The $G$-orbit through each point in $R$ is transverse to $R$.
\end{TL}
\begin{proof}
At a point in $R \cap L_i$, for some $i = 1,\ldots, k$, the local model constructed in Lemma \ref{Lemma: Riemann surface from R_L} shows that $R$ intersects $L_i$ transversely. The closure of the $G$-orbit of such a point is precisely $L_i$ and therefore, the action of $G$ is transverse to $R$.

Suppose that $r \in R$ is not in $L_1, \ldots, L_k$, and choose $L \in \mo$ with $r \notin H.L$. If we choose $w_0 \in G.r \cap L$, then $w \mapsto (w, \chi(w))$ is a biholomorphism from a neighbourhood of $w_0$ in $L$ to a neighbourhood of $r$ in $R$, where we are using the local coordinates about $r$ in $Z$ from Lemma \ref{Lemma: nu definition}. Thus, the tangent plane to $R$ at $r$ is spanned by $$(1, \frac{\partial\chi}{\partial w}(w_0)),$$ which is evidently transverse to the tangent plane to the $G$-orbit through $r$.
\end{proof}

In the next proposition we establish the microtwistor correspondence referred to in the introduction.

\begin{TP}\label{Proposition: Conformal map}
For each $(a, \bar{a}^{-1}) \in \Delta$ there exists a unique $n \in \n$, such that the $G$-orbit through $a^{1/2} \in L_0$ intersects with $n$ tangentially. The induced map from $\Delta$ to $N$ is a conformal isometry.
\end{TP}
\begin{proof}
The map from $\Delta$ to $\n$ in the statement of the proposition is $\sigma\circ\nu$. By making the identification $\n\cong N$, we may consider $\sigma\circ\nu$ as a map to $N$. In the previous section we constructed a conformal map from $\Delta^\circ\times F$ to $M^\circ$, which induces a conformal isometry between $\Delta^\circ$ and $N^\circ$. This map is equivalent to $\sigma\circ\nu$ restricted to $\Delta^\circ$. We know that $\nu$ is a conformal isometry, by Lemma \ref{Lemma: nu definition}. Therefore, $\sigma$ restricts to a conformal isometry between $D^\circ$ and $N^\circ$.

In the previous lemma we established that $D$ is transverse to each $F$-orbit of twistor lines. Therefore, since $\sigma: D \rightarrow \n$ is injective (by Lemma \ref{Lemma: D intersection}), $\sigma$ is a conformal isometry onto its image. Moreover, by the remarks above we know this image contains $\n^\circ$. Thus, to complete the proof, it suffices to show that $\sigma$ restricts to a diffeomorphism between boundaries. By Lemma \ref{Lemma: Delta properties}, $\sigma(\partial D) \subset \partial \n$ and so, there is an induced map between boundaries. Both boundaries are diffeomorphic to $S^1$ and the map between the boundaries is a diffeomorphism onto its image. Therefore, $\sigma$ restricts to a diffeomorphism between boundaries.
\end{proof}



\begin{TR}\label{Remark: S Definition}A conformal structure with torus symmetry is referred to as \emph{surface-orthogonal} if the orthogonal distribution to the torus orbits is integrable. In \cite{J95} Joyce gives a local classification of anti-self-dual conformal structures with torus symmetry that are surface-orthogonal. Using equation (\ref{Equation: real conformal structure}), we can observe that the conformal structure on $M^\circ$ must fit into this classification. However, the existence of $S : = \pi(R)$ is also equivalent to the surface-orthogonality. To see this, note that $R$ intersects a twistor line at the points where the action of $G$ is tangential to that line; therefore, the orthogonal distribution to the $F$-orbits in $M$ restricts to the tangent space of $S$. Thus, $S$ can be obtained by integrating this distribution.\end{TR}

\section{Uniqueness of the conformal structure}\label{Section: Uniqueness}
\subsection{The conformal data}\label{subsection: Determining psi}
Recall that the quotient map from $M$ to $N$ maps the fixed points $x_i$ to $\zeta_i \in \partial N$. We will denote the ordered set $$\{\zeta_1,\ldots, \zeta_k\} \subset \partial N$$ by $\mc{R}$. The conformal isometry between $\Delta$ and $N$, defined in Proposition \ref{Proposition: Conformal map}, maps the ordered set $$\{(z_1^2, z_1^2), \ldots, (z_k^2, z_k^2)\} \subset \partial \Delta$$ to $\mc{R}$. So, we have shown that $\mc{R}$ determines the special points $\pm z_1, \ldots, \pm z_k$ up to sign. In Lemma \ref{Lemma: Special point ordering} we remove the ambiguity in the choice of sign. Therefore, by Remark \ref{Remark: psi determines conformal}, the conformal structure on $M^\circ$ is determined by $\mc{R}$ and the combinatorial data $\mc{T}$, which was defined in \ref{Definition: Combinatorial data}.

We will refer to $\{w \in L: |w| = 1\}$ as the \emph{equator} of $L \in \mo$. So the equator is the set of points in $L$ on special orbits or $B$-orbits. The double-valued map from $L_0$ to $L$ defined in Subsection \ref{Subsection: Coordinates} restricts to a diffeomorphism between the equators on either branch. One of these branches maps $z_i \in L_0$ to $w_i \in L$: the special points $\pm w_i$ have the same ordering along the equator irrespective of the choice of $L$.

\begin{TL}\label{Lemma: Special point ordering}
If $L \in \mo$, the special points $w_1, \ldots, w_k$ satisfy
$$0 = \tn{arg}(w_1) < \tn{arg}(w_2) < \ldots < \tn{arg}(w_k) < \pi.$$
\end{TL}
\begin{proof}
The coordinate on $L \in \mo$, which was determined in Lemma \ref{Lemma: Unique coordinate}, satisfies $w_1 = 1$ and $\tn{arg}(w_2) \in (0, \pi)$. So, we know that $0 =  \tn{arg}(w_1) < \tn{arg}(w_2) < \pi$. If $k = 2$ the proof is complete, so we will assume $k > 2$ and $\pi < \tn{arg}(w_j) < 2\pi$ for some $j \neq 1,2$, in order to obtain a contradiction.

If $L_\infty \in \m$ is the twistor line of some non-fixed point in $B_1 \subset M$, then, by Proposition \ref{Proposition: Conformal map}, $L_\infty$ corresponds to some $(a_\infty, a_\infty) \in \partial\Delta$. Recall from Section \ref{Section: Local model} that $L_\infty$ intersects $C_1^+$, which is a twistor lift of $B_1$, at a single point. We will denote this point by $x^+$. By Lemma \ref{Lemma: Special orbit decomposition}, the curve $C_1^+$ is contained in the closure of the special orbits $O_2^+, \ldots, O_k^+, O_1^-$. Therefore, these special orbits intersect $L_{\infty}$ at $x^+$. Similarly, $L_{\infty}$ intersects the special orbits $O_2^-, \ldots, O_k^-, O_1^+$ at a single point $x^- \in C_1^-$.

It follows from these observations, that $-w_{2}, \ldots, -w_k$ on $L := \Psi(a, b, c) \in \mo$ converge to $w_1 = 1$ as $a \rightarrow a_\infty$, while $w_{2}, \ldots, w_k \rightarrow -1$. By our assumption, one of the sectors between $\pm 1$ on the equator of $L$ contains both $w_{2}$ and $-w_{j}$. Moreover, the ordering of $w_1^2, \ldots, w_k^2$ implies that $w_{2}$ lies in the sub-sector between 1 and $-w_{j}$. This is true for each $L$ irrespective of the choice of $a$. Therefore, we obtain a contradiction since, $w_{2} \rightarrow -1$ and $-w_{j} \rightarrow 1$ as $a\rightarrow a_\infty$ yet $|w_2| = |w_j| = 1$.
\end{proof}

\subsection{Uniqueness of the conformal structure}\label{subsection: uniqueness}
We remarked in the previous subsection that the conformal structure on $M^\circ$ is uniquely determined by $\mc{R}$ and $\mc{T}$. Recall from Definition \ref{Definition: Combinatorial data}, that there is a bijective correspondence between $\s$, which is the combinatorial data associated with $B \subset M$, and $\mc{T}$. In the next lemma we show that $\mc{R}$ and $\mc{S}$ uniquely determine the conformal structure on $M$.

\begin{TL}\label{Lemma: Uniqueness}
Suppose that $(\tilde M, [\tilde g])$ is a compact toric anti-self-dual orbifold containing a surface $\tilde B$ associated with the combinatorial data $\s$ in the way described in Subsection \ref{subsection: constructing B}. Let $\tilde N$ be the $F$-orbit space of $\tilde M$ and $\tilde{\mc{R}} = \{\tilde \zeta_1, \ldots, \tilde \zeta_k \}$ be the ordered set of fixed orbits in $\partial \tilde N$. If there exists a conformal isometry between $\alpha: N \rightarrow   \tilde N$ mapping $\zeta_i$ to $\tilde \zeta_i$ (for $i = 1, \ldots k$), then there is a conformal isometry between $(M, [g])$ and $(\tilde M, [\tilde g])$.
\end{TL}
\begin{proof}
The microtwistor correspondence, defined in Proposition \ref{Proposition: Conformal map}, together with $\alpha$ define a conformal isometry between $\Delta$ and $\tilde \Delta$. As noted above, the conformal structure on $\Delta^{\circ}\times F$, which we constructed in Section \ref{Section: The conformal structure}, is determined by $\s$ and $\mc{R}$. Similarly, the conformal structure on $\tilde \Delta^{\circ}\times F$ is determined by $\s$ and $\tilde{\mc{R}}$, by the conditions of the lemma.  Therefore, the conformal isometry between $\Delta$ and $\tilde \Delta$ induces a conformal isometry $$\beta: \Delta^{\circ}\times F \rightarrow \tilde \Delta^{\circ}\times F.$$

Now note that $\alpha$ preserves orbit type: orbits in $\partial N$ corresponding to $u_i \in \s$ are mapped to orbits in $\partial \tilde  N$ corresponding to $u_i \in \s$. Consequently, $\alpha$ is a weighted diffeomorphism in the terminology of Subsection \ref{Subsection: M simply connected} and, by using the microtwistor correspondence to identify $N$ and $\tilde N$ with $\Delta$ and $\tilde \Delta$ respectively, $\beta$ extends to an $F$-equivariant diffeomorphism between $M$ and $\tilde M$. The conformal structure $[h] := \beta^*[\tilde g]$ on $M$ satisfies $$[h]|_{M^\circ} \equiv [g]|_{M^\circ}$$ and so, $[h]$ and $[g]$ agree on a dense open subset of $M$. Therefore, they define the same conformal structure and $\beta$ is a conformal isometry.
\end{proof}

The previous lemma shows that there is a family of anti-self-dual structures on $M$ parameterized by the $k$ distinct points in $\mc{R} \subset \partial N$. Two sets of conformal data $\mc{R}$ and $\tilde{\mc{R}}$ define the same conformal structure if and only if they are related by a conformal isometry. So, we will denote the anti-self-dual orbifold corresponding to $\mc{R}$ and $\s$ by $\mrs$. The next proposition summarizes our results (and includes part (i) of Theorem A).

\begin{TP}\label{Proposition: Main proposition}
Let $M$ be a compact toric 4-orbifold with positive orbifold Euler characteristic. If $M$ admits an $F$-invariant anti-self-dual conformal structure, then there is a conformal isometry between $M$ and $\mrs$, for some $\s \subset \Lambda$ and $\mc{R} \subset \partial N$.
\end{TP}

\section{The proof of Theorems A}\label{Section: Theorems A}
\subsection{Joyce's construction}\label{Subsection: Joyce's construction}
The conformal coordinate change $$a\mapsto \zeta := -i\frac{a+1}{a-1},$$ identifies $N$ with the closure of $\mc{H}^2$ in $\cp^1$, where $\mc{H}^2 := \{\zeta \in \C: \tn{Im}(\zeta) > 0\}$. In particular the fixed orbit $\zeta_1 \in \mc{R}$, which is identified with $z_1^2 = 1$ in the unit disc model, is identified with $\infty\in\cp^1$. In these coordinates, $P$ and $Q$, which were defined in Subsection \ref{subsection: Conformal structure}, can be written as
\begin{equation}\label{Equation: Joyce solution}Q + iP = \frac{1}{4\pi}\sum_{j=1}^{k} \Big(\frac{\zeta - \zeta_j}{\bar{\zeta} - \zeta_j}\Big)^{1/2}v_{j}.\end{equation}
When we set $\zeta = x + iy$, this can be written in the form found in \cite{J95} as \[ \frac{1}{4\pi} \sum_{i=1}^k f^{\zeta_i}\; v_i, \] where \begin{equation}\label{Equation: Elementary Joyce solution} f^{\zeta_i}(x,y) := \big((x-\zeta_i) + iy\big)\big((x-\zeta_i)^2 + y^2\big)^{-1/2}.\end{equation} Furthermore, the representative of the conformal structure in (\ref{Equation: real conformal structure}) transforms to
\begin{equation}\label{Equation: Joyce conformal structure} \frac{\tn{d}x^2 + \tn{d}y^2}{y^2} + \frac{(P_2 \tn{d}\theta_1 - P_1 \tn{d}\theta_2)^2+(Q_2 \tn{d} \theta_1 - Q_1 \tn{d} \theta_2)^2}{(P_1Q_2 - P_2Q_1)^2}.\end{equation}

In Theorem 3.3.1 of \cite{J95} Joyce proves that, when $M_\s\cong k\overline{\cp}^2$, the conformal structure (\ref{Equation: Joyce conformal structure}) on $\mc{H}^2\times F$ extends to a conformal structure on $M_\s$. The proof of Theorem 3.3.1 of \cite{J95} involves showing that \begin{equation}\label{Equation: Boundary non-degenacy}\frac{\delta}{y}(P_1Q_2 - P_2Q_1) > 0\end{equation} is satisfied on $\mc{H}^2$ and along $\partial \mc{H}^2$, where $\delta.\big((x-\zeta_i)^2 + y^2\big)^{-1/2}$ is bounded and non-zero, for $i = 1,\ldots, k$. However, provided $\s$ is chosen so that $P$ and $Q$ satisfy $$P_1Q_2 - P_2Q_1 > 0,$$ his proof generalizes in a straightforward manner to construct conformal structures on $M_\s$, as noted in \cite{CS04}. In Lemma \ref{Lemma: non-degeneracy} we proved that $P_1Q_2 - P_2Q_1 \neq 0$. If $P_1Q_2 - P_2Q_1 < 0$, then the conformal structure on $\mc{H}^2\times F$ would be self-dual, since reversing the orientation on $F$ would produce an anti-self-dual conformal structure. Therefore, $P_1Q_2 - P_2Q_1 > 0$ and so, $\mrs$ is conformally equivalent to a metric arising from Joyce's construction, by Proposition \ref{Proposition: Main proposition}. This completes the proof of part (iii) of Theorem A.

\begin{TR}\label{Remark: Joyce error}
We remark that the solution to Joyce's equation (\ref{Equation: Joyce solution}), which we use to construct a conformal structure on $M$, does \emph{not} agree with the solution written down by Joyce in equation (44) of \cite{J95} that is intended for the same purpose. This is due to a minor error in \cite{J95}. However, if we define $u^{\perp} \in \mf{f}^*$ to be the rotation by $-\pi/2$ of the dual vector to $u \in \mf{f}$, then we obtain (44) of \cite{J95}, by replacing the generator of the stabilizer subgroup $u_i$ in (\ref{Equation: Joyce solution}) by its annihilator $u_i^\perp$.
\end{TR}

\subsection{Negative-definite intersection form}\label{Subsection: Negative-definiteness}
To complete the proof of Theorem A, we must now prove part (ii): that $M$ has \emph{negative-definite intersection form}. We will denote $\ms\backslash\{x_1\}$ by $Y$. In \cite{J95}, Joyce showed that the conformal structure on $\mrs$ restricted to $Y$ contains a \emph{scalar-flat K\"ahler} representative $g_0$, which is obtained by multiplying (\ref{Equation: Joyce conformal structure}) by $$y.(P_1Q_2 - P_2Q_1).$$ We note that the associated complex structure on $Y$ corresponds to the divisor $$(O_1^+ \cup O_1^-) \cap (Z\backslash \{L_1\}).$$

\begin{figure}[htbp]
  \psfrag{p1}{$\zeta_1$}
  \psfrag{p2}{$\zeta_2$}
  \psfrag{p3}{$\zeta_3$}
  \psfrag{pk}{$\zeta_k$}
  \psfrag{inf}{$\zeta_1 = \infty$}
  \psfrag{q2}{$\zeta_2$}
  \psfrag{q3}{$\zeta_3$}
  \psfrag{qk}{$\zeta_k$}
  \psfrag{wk}{$u_k$}
  \psfrag{cg}{$[g]$}
  \psfrag{v1}{$u_1$}
  \psfrag{v2}{$u_2$}
  \psfrag{vk}{$u_k$}
  \psfrag{w1}{$u_1$}
  \psfrag{w2}{$u_2$}
  \psfrag{wk}{$u_k$}
  \psfrag{g}{$\scriptstyle g\;=\;V(dx^2 + dy^2) + V^{-1}\big((P_1 d\theta_1 - P_2 d\theta_2)^2 + (Q_1 d\theta_1 - Q_2 d\theta_2)^2\big)$}
  \centering \scalebox{0.8}{\includegraphics{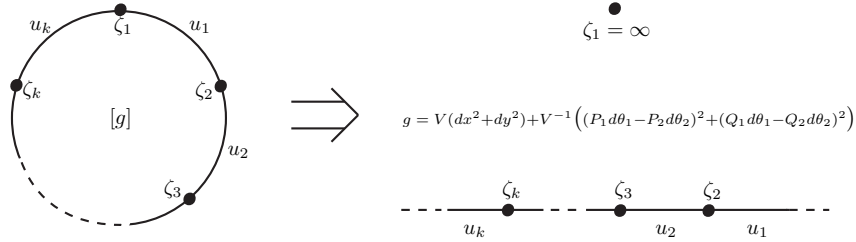}}
  \caption{Choosing a scalar-flat K\"ahler structure}
\end{figure}

In the next lemma we establish negative definiteness of the intersection form by equating it with \emph{convexity} of the moment polytope.

\begin{TL}\label{Lemma: Intersection form}
The intersection form on $M$ is negative-definite.
\end{TL}
\begin{proof}
In Subsection \ref{Subsection: Boundary orbits and fixed points} we confined $\s$ to the $180^\circ$ sector in $\mf{f}$ by requiring $u_i$ to satisfy either $u_i.(0,1) > 0$ or $u_i = (p,0)$, for $p > 0$. Then in Subsection \ref{subsection: constructing B} we fixed the coordinates on $F$ so that $u_1 = (p,0)$.

The K\"ahler form on $(Y,g_0)$ is given by $$\tn{d}x \wedge (Q_2\tn{d}\theta_1 - Q_1\tn{d}\theta_2) + \tn{d}y \wedge (P_2\tn{d}\theta_1 - P_1\tn{d}\theta_2).$$
Therefore, the K\"ahler structure together with the action of $F$ define a \emph{moment map} $\mu: Y \rightarrow \mf{f}^*$, which can be written as
\begin{equation}\label{Equation: moment map} x.v_1^\perp + \sum_{i=2}^{k} \big((\zeta - \zeta_i)(\bar{\zeta} - \zeta_i)\big)^{1/2}.v_i^{\perp},\end{equation} where $\perp$ was defined in Remark \ref{Remark: Joyce error}. The preimage of compact subsets of $\mf{f}^*$ are compact subsets of $Y$ and so, $\mu$ is proper. This provides a sufficient condition for the moment polytope $\mu(Y)$ to be convex, as shown in \cite{LMTW98}. The set $B$ is mapped to the boundary of this polytope: the surfaces $B_i$ are mapped to edges; the fixed points $x_2, \ldots, x_k$ are mapped to vertices; and $x_1$ is mapped to $\infty$. It follows from equation (\ref{Equation: moment map}) that the edge $\mu(B_i)$ has slope $u_i^{\perp} \in \Lambda^*$.

With the choice of coordinates we have made for the action of $F$, convexity implies that $u_i$ is the outward pointing normal to the face $\mu(B_i)$ of the moment polytope in $\mf{f}^*$, or equivalently $$u_i^{\perp}. u_{i-1} < 0,$$ for $i = 1, \ldots, k$. It was proven, in the smooth case by Joyce \cite{J95} and in the orbifold case by Calderbank and Singer \cite{CS06}, that this condition is equivalent to the intersection form being negative-definite.
\end{proof}

\subsection{The proof of Theorem B}\label{Subsection: Theorem B}
Using Theorem A it is a simple task to classify \emph{asymptotically locally Euclidean} (ALE) scalar-flat K\"ahler toric 4-orbifolds, following the procedure suggested by Chen, LeBrun and Weber in \cite{CLW07}.

\begin{TD}[Joyce \cite{J00}]\label{Definition: ALE}
For a discrete subgroup $\Gamma \subset SO(n)$, the smooth orbifold $\mathbb{R}^n/\Gamma$ has a Riemannian (orbifold) metric induced from the standard Euclidean metric on $\mathbb{R}^n$. We will denote this metric by $g$ and its Levi-Civita connection by $\nabla$.

Let $(M^n,h)$ be a non-compact Riemannian orbifold. If there exists a compact set $K \subset M$ such that:
\begin{itemize}
\item{there are finitely many connected components of $M\backslash K$;}
\item{for some $R > 0$, there exists a diffeomorphism $a$ from each connected component onto $$\{x\in \mathbb{R}^n: |x| > R  \}/\Gamma;$$}
\item{and, for some $l>0$ and all $m \in \mathbb{Z}_{\geq 0}$ the push-forward metric $a_*(h)$ satisfies  $$|\nabla^m (a_*(h) - g)| = O(|x|^{-m-l});$$}
\end{itemize}
then $(M, h)$ is ALE to order $l$.
\end{TD}

We will suppose that $(X,g)$ is a scalar-flat K\"ahler toric 4-orbifold that is ALE to order $l > 3/2$. Note that the scalar-flat K\"ahler metrics constructed in Subsection \ref{Subsection: Negative-definiteness} satisfy this property, since they are ALE to order 2, as shown in \cite{RS05}.

A K\"ahler metric is anti-self-dual if and only if the scalar curvature vanishes; thus, $g$ is an anti-self-dual orbifold metric. A single point can be added to each ALE end of $X$ to produce a compact toric 4-orbifold $M$ with a conformal structure $[g]$ extending $g$. Each of these additional points cannot belong to any of the orbits in $X$ and so, they must be fixed by the torus action. By the Poincar\'{e}-Hopf theorem for orbifolds (given in Subsection \ref{Subsection: Tensors}), each fixed point makes a positive contribution to the Euler characteristic and so, $\chi_{orb}(M)>0$. Proposition 12 of \cite{CLW07} states that an anti-self-dual metric, which is ALE to order $l > 3/2$, extends to an anti-self-dual conformal structure on the orbifold compactification. Therefore, Theorem A can be applied to $M$. Then, without loss of generality, we can assume that $x_1 \in M - X$. On $Y := M\backslash \{x_1\}$ there is a unique K\"ahler representative $g_0 \in [g]$ and therefore, $(X, g) = (Y, g_0)$. This proves Theorem B, which we restate below in a more compact form.

\begin{TheoremB}
Let $X$ be a K\"ahler toric 4-orbifold that is scalar-flat and ALE to order $l > 3/2$. Then, up to homothety, there is a torus equivariant isometry between $X$ and the scalar-flat K\"ahler representative in $\mrs\backslash\{x_1\}$, for some $\mc{R}$ and $\s$.
\end{TheoremB}

Now we assume that $X$ is smooth; accordingly, coordinates can be chosen for the action of $F$ on $M$ so that $u_1 = (0,1)$, $u_{2} = (1,0)$ and $u_k = (p,q)$, where $p$ and $q$ are positive and coprime. In these coordinates, the restrictions imposed on $\s$ by convexity of the moment polytope (cf. Lemma \ref{Lemma: Intersection form}) are equivalent to those imposed on an \emph{admissible sequence}, in the sense of \cite{CS04}. It follows that the metric we have constructed on $X$ belongs to the family of metrics constructed by Calderbank and Singer in \cite{CS04}. This proves Corollary \ref{Corollary: Smooth SFK}.

\begin{TR}
The surface-orthogonal surface $S$, which was defined in Remark \ref{Remark: S Definition}, can be constructed by gluing together four copies of $N$ along $\partial N$ according to the data $\mc{R}$ and $\s$. Thus, the anti-self-dual structures on $M_\s$ can be parameterized by the conformal structures on a surface constructed in this way, rather than by $\mc{R}$. The same can be said about the ALE scalar-flat K\"ahler structures on $X$. In this case, $S$ can be identified with the extension to $X$ of the real surface $$\R^2 \cap (\C^*\times\C^*) \subset X.$$ 
\end{TR}

\end{document}